\documentclass[a4paper]{article}

\usepackage{epsfig}
\usepackage{graphicx}
\usepackage{amsmath}
\usepackage{amsfonts}
\usepackage{amssymb}
\usepackage{amsthm}

\newtheorem{thm}{Theorem}[section]

\newtheorem{pro}{Proposition}[section]
\newtheorem{lem}[thm]{Lemma}
\newtheorem{rem}[thm]{Remark}

    \newcommand{\nrm}[1]{\left\| #1 \right\|}

\newcommand{\qref}[1]{(\ref{#1})} 

    \def\x{\mbox{\boldmath $x$}}
    \def\dt{\Delta t}
    \def\exac{{\mbox \tiny e}}

\def\dspace{\displaystyle \vspace{.05in}}

\begin{document}
\title{\textbf{A Fourier Pseudospectral Method 
for the ``Good" Boussinesq Equation with Second Order Temporal Accuracy}}

\author{\normalsize  {Kelong Cheng$^1$}, {Wenqiang Feng$^{2}$}, {Sigal Gottlieb$^3$}, {Cheng Wang$^3$}\\
\normalsize{  $^1$School of Science, Southwest University of Science and Technology,}\\
\normalsize { Mianyang, Sichuan 621010, P. R. China} \\
\normalsize{ $^2$Department of Mathematics, University of Tennessee, Knoxville, TN 37996, USA} \\
\normalsize $^3$Mathematics Department,  University of Massachusetts, North Dartmouth, MA 02747, USA.}

\date{}
\maketitle

\begin{abstract}
In this paper, we discuss the nonlinear stability and convergence of a fully discrete 
Fourier pseudospectral method coupled with a specially designed second order 
time-stepping for the numerical solution of the ``good" Boussinesq equation. 
Our analysis improves the existing results presented in 
earlier literature in two ways. First, an $\ell^\infty (0,T^*; H^2)$ 
convergence for the solution and $\ell^\infty (0,T^*; \ell^2)$ convergence for the time-derivative 
of the solution are obtained in this paper, instead of the $\ell^\infty (0,T^*; \ell^2)$ 
convergence for the solution and the $\ell^\infty (0,T^*; H^{-2})$ convergence for the time-derivative, 
given in \cite{Fru}. In addition, the stability and convergence of this method is shown to be unconditional 
for the time step in terms of the spatial grid size, compared with a severe restriction time step restriction
$\dt  \le C h^2$ reported in \cite{Fru}.

\end{abstract}


\maketitle

\allowdisplaybreaks

%%%%%%%%%%%%%%%%%%%%%%%%%%%%%%%%%%%%%%%%%%%%%%%%%%%%%%%%%%%%%%%%%%%%%%%%%%%%%%%%%%%%%%%%%%%%%%%%%%%%%%%%%%%%

\section{Introduction}
\setcounter{page}{1}

The soliton-producing nonlinear wave equation is a topic of significant scientific interest.
One commonly used example  is  the so-called ``good" Boussinesq (GB) equation
\begin{equation}
u_{tt}=-u_{xxxx}+u_{xx}+(u^p)_{xx} ,  \quad 
  \mbox{with an integer $p \ge 2$} .   
\label{bsq}
\end{equation}
%which is typically called the ``good" Boussinesq (GB) equation.  
It is similar to the well-known Korteweg-de Vries (KdV) equation;  
a balance between dispersion and nonlinearity leads to the existence of solitons. The GB equation and its various extensions have been investigated by many authors. For instance, a closed form
solution for the two soliton interaction of Eq. \qref{bsq} was obtained by Manoranjan {\it{et al.}} in \cite{Man1} and a few numerical experiments were performed based on the
Petrov-Galerkin method with linear ``hat" functions. In \cite{Man2}, it was shown that the GB equation possesses
a highly complicated mechanism for the solitary waves interaction. Ortega and  Sanz-Serna \cite{Ort} 
discussed  nonlinear stability and
convergence of some simple finite difference schemes for the numerical solution of this equation. 
%Frutos  {\it{et al.}} \cite{Fru} considered  a pseudospectral method for the periodic problem of GB equation. 
More analytical and numerical works related to GB equations can be found in the literature, for example 
\cite{Att, Bra1, Bra2, Cie1, Cie2, Farah10, Fru, Got, Linares1995, Oh13, Pan, TM1991}.

In this paper, we consider the GB equation \qref{bsq},  with a periodic boundary condition 
over an 1-D domain $\Omega = (0,L)$ and initial data 
$u(x,0)=u^0(x)$, $u_t(x,0)=v^0(x)$, both of which are $L$-periodic. 
%\begin{equation}\label{bsq}
%\left\{\begin{split}
%&u_{tt}=-u_{xxxx}+u_{xx}+(u^2)_{xx} , \quad  \quad 0 \le x \le L , \, \, 
%  0\leq t <T ,\\
%&u(x,t)=u(x+L,t) ,  \quad  \forall x \in R , \, \, 0\leq t \leq T,\\
%&u(x,0)=u^0(x),\quad  0 \le x \le L ,  \\
%&u_t(x,0)=v^0(x),\quad  0 \le x \le L ,
%\end{split}\right.
%\end{equation}
%where the data $u^0(x)$, $v^0(x)$ are $L$-periodic function, 
It is assumed that a unique, periodic, smooth enough solution exists for (\ref{bsq}) 
over the time interval $(0,T)$. This
$L$-periodicity assumption is reasonable if the solution to \qref{bsq} decays exponentially  outside $[0,L]$. 

Due to the periodic boundary condition, the Fourier collocation (pseudospectral) differentiation is a natural choice to obtain the optimal
spatial accuracy.  
There has been a wide and varied literature on
the development of spectral and pseudospectral schemes. 
For instance, the stability analysis for linear time-dependent problems
can be found in \cite{GO1977, MMO1978}, etc, based on 
eigenvalue estimates. 
Some pioneering works for nonlinear equations were initiated by 
Maday and Quarteroni \cite{MQ1981, MQ1982a, MQ1982b} for steady-state spectral solutions.  Also note the analysis of one-dimensional 
conservation laws by Tadmor and collaborators  \cite{CDT1993,GTM2001,
MOT1993,ET1989,ET1990,ET1993,ET2004}, semi-discrete 
viscous Burgers' equation and Navier-Stokes equations by E \cite{WE1992,WE1993}, 
the Galerkin spectral method for Navier-Stokes equations led by Guo \cite{Guo1995, GH2007, 
GLM1995, GTM2001} and Shen \cite{DGS2001, GS2001}, 
and the fully discrete (discrete both  in space and time)
pseudospectral method applied to viscous Burgers' equation 
in \cite{Got} by Gottlieb and Wang and \cite{BQ1986} 
by Bressan and Quarteroni, etc. 

%  Among the existing articles for fully discrete (discrete both in space and time)
%pseudospectral method applied to nonlinear equations, it is worth 
%mentioning a recent work \cite{Got} 
%on the three-dimensional viscous Burgers' equation, in which an 
%$l^\infty (0,T^*; H^2)$ convergence is established. The key point is this 
%analysis is an aliasing error control estimate, which enables the authors to 
%derive an $l^\infty (0,T^*; H^2)$ error estimate for a fixed final time. 
%In turn, the maximum norm bound of the numerical solution, which 
%has always been the essential difficulty in the nonlinear numerical analysis, 
%could be automatically 
%obtained, because of the $H^2$ error estimate and the corresponding Sobolev 
%embedding in 3-D. As a result, this approach avoids the standard use of the inverse inequality and leads to unconditional convergence (for $\dt$ independent on the space grid size $h$), which turns out to be an 
%improvement over an earlier work \cite{BQ1986}. 

Most of the theoretical developments in  nonlinear spectral and pseudospectral schemes are related to a 
parabolic PDE, in which the diffusion term plays a key role in the 
stability and convergence analysis. 
Very few works have analyzed a fully discrete 
pseudospectral method applied to a nonlinear hyperbolic PDE. 
Among the existing ones, it is worth mentioning Frutos  {\it{et al.}}'s 
work \cite{Fru} on the nonlinear analysis of a second order 
(in time) pseudospectral scheme for the GB equation (with $p=2$).
However, as the authors point out in their remark on page 119, these theoretical results were not optimal:
 {\it{`` ... our energy norm is an $L^2$-norm of $u$ combined with a negative norm of $u_t$. 
 This should be compared with the energy norm in \cite{LS1988}: there, no integration with respect to $x$ is 
necessary and convergence is proved in $H^2$ for $u$ and $L^2$ for $u_t$''}}.
The  difficulties in the analysis  are due to the  absence of 
a dissipation mechanism in the GB equation (\ref{bsq}), which makes the nonlinear error terms much 
more challenging to analyze than that of a parabolic equation. 
The presence of a second order spatial derivative 
for the nonlinear term leads to an essential difficulty of numerical 
error estimate in a higher order Sobolev norm. 
In addition to the lack of optimal numerical error estimate, \cite{Fru} 
also imposes a severe time step restriction: $\dt \le C h^2$ (with $C$ 
a fixed constant), in the nonlinear stability analysis. Such a constraint 
becomes very restrictive for a fine numerical mesh and  
leads to a high computational cost.

In this work we propose a second order (in time) 
pseudospectral scheme for the the GB equation (\ref{bsq}) with an 
alternate approach, and 
provide a novel nonlinear analysis. In more detail, an $\ell^\infty (0,T^*; H^2)$ 
convergence for $u$ and $\ell^\infty (0,T^*; \ell^2)$ convergence for 
$u_t$ are derived, compared with the $\ell^\infty (0,T^*; \ell^2)$ convergence 
for $u$ and $\ell^\infty (0,T^*; H^{-2})$ convergence for $u_t$, 
as reported in \cite{Fru}. Furthermore, such a convergence is 
unconditional (for the time step $\dt$ in terms of space grid size $h$) 
so that the severe time step constraint $\dt \le C h^2$ is avoided. 

  The methodology of the proposed second order temporal discretization 
is very different from that in \cite{Fru}. To overcome the difficulty 
associated with the second order temporal derivative in the hyperbolic 
equation, we introduce a new variable $\psi$ to approximate $u_t$, 
which greatly facilitates the numerical implementation.  On the other hand, the corresponding second order consistency analysis becomes non-trivial 
because of an $O(\dt^2)$ numerical error between the centered difference of $u$ and the mid-point average of $\psi$. Without a careful treatment, 
such an $O (\dt^2)$ numerical error might seem to introduce a reduction of temporal accuracy, because of the second order time derivative involved in the equation.  To overcome this difficulty, we perform a higher order consistency analysis by an asymptotic expansion; as a result, the constructed approximate solution satisfies the numerical scheme with a higher order truncation error.  
Furthermore, a projection of the exact solution onto the Fourier 
space leads to an optimal regularity requirement. 

For the nonlinear stability and convergence analysis, we have to obtain a 
direct estimate of the (discrete) $H^2$ norm of the nonlinear numerical 
error function. This estimate relies on the
aliasing error control lemma for pseudospectral approximation to 
nonlinear terms, which was proven in a recent work \cite{Got}. 
That's the key reason why we are able to overcome the key difficulty 
in the nonlinear estimate and obtain an $\ell^\infty (0,T^*; H^2)$ 
convergence for $u$ and $\ell^\infty (0,T^*; \ell^2)$ convergence for $u_t$. 
We prove that  the proposed numerical scheme 
is fully consistent (with a higher order expansion), stable 
and convergent in the $H^2$ norm up to some fixed final time $T^*$.  
In turn,  the maximum norm bound of the numerical solution is 
automatically obtained, because of the $H^2$ error estimate and the corresponding Sobolev embedding. Therefore, the inverse inequality in 
the stability analysis is not needed and any scaling law between 
$\dt$ and $h$ is avoided, compared with the $\dt \le C h^2$ constraint 
reported in \cite{Fru}. 

% Hence, our work is to improve the results In \cite{Fru} and obtain the optimal estimate. In this work, we adopt a novel stability and convergence analysis for the pseudospectral method proposed by Gottlieb and Wang \cite{Got},\cite{Lowengrub}, coupled with a number of carefully tailored time discretization for the equation. For each time-discretization, we prove that the equation is solved by the pseudospectral up to some fixed final time $T$, the method is consistent, stable and convergent in the $H^2$-norm for $u$ and $L_2$ for $u_t$.

This paper is outlined as follows. In Section 2 we review the
Fourier spectral and pseudospectral differentiation, recall an aliasing error control lemma (proven in \cite{Got}), and present an alternate second 
order (in time) pseudospectral scheme for the GB equation (\ref{bsq}). 
 In Section 3, the consistency analysis of the scheme is studied in detail.  
The stability and convergence analysis is reported in Section 4. 
A simple numerical result is presented in Section 5. 
Finally, some concluding remarks are made in Section 6.

%%%%%%%%%%%%%%%%%%%%%%%%%%%%%%%%%%%%%%%%%%%%%%%%%%%%%%%%%%%%%%%%%%%%%%%%%%%%%%%%%%%%%%%%%%%%%%%%%%%%%%%%%%%%%%%%%%%

\section{The Numerical Scheme}

\subsection{Review of Fourier spectral and pseudospectral approximations}

For $f(x) \in L^2 (\Omega)$, $\Omega= (0,L)$,  with Fourier series  
\begin{eqnarray} 
 f(x) =   \sum_{l=-\infty}^{\infty}  
 \hat{f}_l {\rm e}^ {2 \pi {\rm i} l x/L },  \quad   
  \mbox{with} \quad 
    \hat{f}_{l}  =  \int_{\Omega}  f(x)  
    {\rm e}^{-2 \pi {\rm i} l x/L } dx , 
\end{eqnarray}
its truncated series is defined as the projection onto the space 
${\cal B}^N$ of trigonometric polynomials in $x$ of degree up to $N$,  
given by 
\begin{equation} 
  {\cal P}_N f(x) =   \sum_{l=-N}^{N}  
 \hat{f}_{l} {\rm e}^ {2 \pi {\rm i} l x/L } .  
\end{equation}

To obtain a pseudospectral approximation at a given set of points, 
an interpolation operator ${\cal I}_N$ is introduced. 
Given a uniform numerical grid with $(2N+1)$ points 
and a discrete vector function ${\mathbf f} $ where ${\mathbf f}_i = f(x_i) $, for each spatial point
 $x_i$.  The Fourier interpolation of the function is defined by 
\begin{equation} 
   \left({\cal I}_N f \right)(x) =   \sum_{l=-N}^{N}   
  (\hat{f}_c^N)_{l} {\rm e}^ {2 \pi {\rm i} l x/L }, \label{spectral-interpolation} 
\end{equation}
 where the $(2N+1)$ pseudospectral coefficients $ (\hat{f}_c^N)_l$ are computed based on  
 the interpolation condition  $  f(x_i) =  \left( {\cal I}_N f \right)(x_i)$
on the $2N+1$ equidistant points \cite{Boyd2001, CHQZ2007, HGG2007}. 
%Sometimes we need to look at the continuous function that 
%results from evaluating the coefficients by collocation. 
%Given a function 
%$f (x)$, we compute its collocation coefficients $(\hat{f}_c)^N_{l}$ based on the $2N+1$ equidistant points so that the function ${\cal I}_N f (x)$ is defined to be the continuous expansion based on these coefficients, given by (\ref{spectral-interpolation}). 
 These collocation coefficients can be efficiently computed using the fast Fourier transform (FFT).
%Note that the pseudospectral coefficients depend on the number of points:  increasing $N$ gives a completely
%different set of coefficients.  
Note that the pseudospectral coefficients are not equal to the actual Fourier coefficients;
 the difference between them is known as the aliasing error.
In general, ${\cal P}_N f(x)  \ne {\cal I}_N f(x) $, and even
${\cal P}_N f(x_i)  \ne {\cal I}_N f(x_i) $,
except of course in the case that $ f \in {\cal B}^N$.

The Fourier series and the formulas for its projection and interpolation 
allow one to easily take derivative by simply multiplying the appropriate Fourier coefficients  $(\hat{f}_c^N)_{l} $ by $2 l \pi {\rm i}/L$.  
Furthermore, we can take subsequent derivatives in the same way, so that
differentiation in physical space is accomplished via multiplication in Fourier space.
As long as $f$ and all is derivatives (up to $m$-th order) are continuous and periodic on $\Omega$, the convergence of the derivatives of the projection and interpolation %converge exponentially to those of the function: 
is given by 
\begin{eqnarray}  
 \|\partial^k f(x) - \partial^k {\cal P}_N f(x) \| 
  &\leq&  C \| f^{(m)} \| h^{m-k} , \quad 
\mbox{for} \, \, \, 0 \le k \le m , \nonumber 
\\
  \|\partial^k f(x) - \partial^k {\cal I}_N f(x) \| 
  &\leq&  C \| f \|_{H^m} h^{m-k} , \quad 
\mbox{for} \, \, \, 0 \le k \le m ,  \, m > \frac{d}{2} ,   
   \label{spectral-approximation}
\end{eqnarray}
in which $\| \cdot \|$ denotes the $L^2$ norm.  
For more details, see the  discussion of approximation theory by Canuto and  Quarteroni \cite{CQ1982} . 

%The definition of pseudospectral differentiation can be outlined as follows. 
For any collocation approximation to the function $f(x)$ at the points $x_i$ 
\begin{equation} 
  f(x_i) = \left( {\cal I}_N f \right)_{i}=  \sum_{l=-N}^{N}  
  (\hat{f}_c^N)_{l} {\rm e}^{ 2 \pi {\rm i} l x_i}, \label{spectral-1}  
\end{equation}
one can define discrete differentiation operator ${\cal D}_{N}$ 
operating on the 
vector of grid values ${\mathbf f} = f(x_i) $. 
In practice, one may compute the collocation coefficients $(\hat{f_c^N)}_{l}$ via FFT, and then multiply them 
 by the correct values (given by $2 l \pi {\rm i}$) and perform the inverse FFT.  
 Alternatively, we can view the differentiation operator  ${\cal D}_{N}$ 
as a matrix, 
and the above process can be seen as a matrix-vector multiplication.
% Once again, we note that the derivative of the interpolation is not the interpolation of the derivative: 
 %${\cal D}_{N} {\cal I}_N f \ne {\cal I}_N {\cal D}_{N}  f$.
The same process is performed for the second and fourth derivatives $\partial_x^2$, $\partial_x^4$, where this time the collocation coefficients are multiplied
by $( - 4 \pi^2 l^2/L^2 )$ and $( 16 \pi^4 l^4/L^4 )$, respectively.  
In turn, the differentiation matrix can be applied for multiple times, 
i.e. the 
vector ${\mathbf f} $ is multiplied by $ {\cal D}^2_{N} $ and 
${\cal D}_N^4$, respectively. 

Since the pseudospectral differentiation is taken at a point-wise level, 
a discrete $L^2$ norm and inner product need to be introduced to facilitate the analysis. 
Given any periodic grid functions ${\mathbf f} $ and ${\mathbf g} $ (over the numerical grid), 
we note that these are simply vectors and define the discrete $L^2$ inner product and  norm 
\begin{eqnarray} 
  \left\| {\mathbf f} \right\|_2 = \sqrt{ \left\langle {\mathbf f} , 
  {\mathbf f}  \right\rangle } ,  \quad \mbox{with} \quad 
  \left\langle {\mathbf f} , {\mathbf g} \right\rangle  
  = \frac{1}{2N+1} \sum_{i=0}^{2 N}   {\mathbf f}_{i} {\mathbf g}_{i} . 
  \label{spectral-inner product-1}
\end{eqnarray}
The following summation by parts (see \cite{Got}) will be of use:
\begin{eqnarray} 
   \left\langle {\mathbf f} , {\cal D}_N {\mathbf g} \right\rangle  
  =  -  \left\langle {\cal D}_N {\mathbf f} , {\mathbf g} \right\rangle  ,  
  \quad 
  \left\langle {\mathbf f} , {\cal D}_N^2 {\mathbf g} \right\rangle  
  =  -  \left\langle {\cal D}_N {\mathbf f} , 
  {\cal D}_N {\mathbf g} \right\rangle  ,   \quad 
   \left\langle {\mathbf f} , {\cal D}_N^4 {\mathbf g} \right\rangle  
  =   \left\langle {\cal D}_N^2 {\mathbf f} , 
  {\cal D}_N^2 {\mathbf g} \right\rangle . \label{spectral-inner product-2} 
\end{eqnarray}

\subsection{An aliasing error control estimate in Fourier pseudospectral 
approximation} 

This  lemma, established in  \cite{Got}, allows us to bound the 
aliasing error for the nonlinear term, and will be critical to our analysis. 
For any function $\varphi(x)$ in the space ${\cal B}^{p N}$,  
its collocation coefficients $\hat{q}^N_{l}$ are computed based on the $2N+1$ equidistant points. In turn, ${\cal I}_N \varphi(x)$ is given by the continuous expansion based on these coefficients:
\begin{equation} 
  {\cal I}_N \varphi(x)= \sum_{l= -N}^N \hat{q}^N_{l} 
  {\rm e}^{2 \pi {\rm i} l x/L } .  
\end{equation} 
Since  $\varphi(x) \in {\cal B}^{p N}$, we have 
${\cal I}_N \varphi(x) \ne {\cal P}_N \varphi(x) $ due to the aliasing error. 

The following lemma enables us to obtain an $H^m$ bound of the 
interpolation of the nonlinear term; the detailed proof can be found in 
\cite{Got}. 
%In fact, the case of $k=0$ was proven in Weinan E's earlier work \cite{}, 
%we extend it to a general integer $k \ge 0$. 

\begin{lem} \label{aliasing error}
  For any $\varphi \in {\cal B}^{p N}$ (with $p$ an integer) in dimension $d$, we have 
\begin{equation} 
  \left\| {\cal I}_N \varphi \right\|_{H^k}  
  \le  \left( \sqrt{p} \right)^d  \left\|  \varphi \right\|_{H^k} . 
   \label{spectral-projection-4}
\end{equation} 
\end{lem}

\subsection{Formulation of the numerical scheme}

%\setcounter{equation}{0}
%The main difficulty encountered lies in the stability analysis of the aliasing error in the nonlinear term. To implement our purpose,
We propose the following fully discrete second order (in time) scheme for the equation (\ref{bsq}): 
\begin{equation}\label{numerical}
\left\{\begin{array}{rl} 
  \dspace 
\frac{\psi^{n+1}-\psi^{n}}{\dt}
=& -D_N^4 \left( \frac{u^{n+1}+u^n}{2} \right)  
  + D_N^2 \left( \frac{u^{n+1}+u^n}{2} \right)  \\
  \dspace 
& +D_N^2\left(\frac{3}{2}(u^n)^p -\frac{1}{2}(u^{n-1})^p \right),  \\
 \dspace 
\frac{u^{n+1}-u^n}{\dt}
=& \frac{\psi^{n+1}+\psi^{n}}{2},
\end{array} \right. 
\end{equation}
where $\psi$ is a second order approximation to $u_t$ 
and $D_N$ denotes the discrete differentiation operator.

\begin{rem} 
  With a substitution $\psi^{n+1} = \frac{2 (u^{n+1} - u^n)}{\dt}
- \psi^n$, the scheme (\ref{numerical}) can be reformulated
as a closed equation for $u^{n+1}$:
\begin{eqnarray}
  \frac{2 u^{n+1} }{\dt^2} 
  + \frac12 \left( D_N^4 - D_N^2 \right) u^{n+1} 
    &=& D_N^2\left(\frac{3}{2}(u^n)^p-\frac{1}{2}(u^{n-1})^p \right)
     - \frac12 \left( D_N^4 - D_N^2 \right) u^n  \nonumber 
\\
  && 
     +  \frac{\frac{2 u^n}{\dt} + 2 \psi^n}{\dt} .
    \label{scheme-2nd-3}
\end{eqnarray}
Since the treatment of the nonlinear term is  fully explicit, 
this resulting implicit scheme requires only a  linear solver. 
Furthermore, a detailed calculation 
shows that all the eigenvalues of the linear operator on the left hand side are positive, 
and so the unique unconditional  solvability of the proposed 
scheme (\ref{numerical}) is assured. 
In practice, the FFT can be  utilized to efficiently  obtain the numerical solutions.
\end{rem}

\begin{rem} 
  An introduction of the variable $\psi$ not only facilitates the numerical 
implementation, but also improves the numerical stability, due to the fact 
that only two consecutive time steps $t^n$, $t^{n+1}$, are involved in 
the second order approximation to $u_{tt}$. In contrast, three time steps 
$t^{n+1}$, $t^n$ and $t^{n-1}$ are involved in the numerical approximation 
to the second order temporal derivative as presented in the earlier 
work \cite{Fru} (with $p=2$): 
\begin{eqnarray} 
   \frac{u^{n+1}- 2 u^{n} + u^{n-1} }{\dt^2}
 =- \frac14 D_N^4 \left( u^{n+1} +2 u^n + u^{n-1}  \right)  
  + D_N^2 u^n +D_N^2\left( (u^n)^2 \right) . 
   \label{2nd order scheme-1991} 
\end{eqnarray}

A careful numerical analysis in \cite{Fru} shows that the numerical stability 
for (\ref{2nd order scheme-1991}) is only valid under a severe time step constraint $\dt \le C h^2$, since this scheme is evaluated at the time step $t^n$. On the other hand, the special structure of our proposed 
scheme (\ref{numerical}) results in an unconditional stability and 
convergence for a fixed final time, as will be presented in later analysis. 
\end{rem}

%%%%%%%%%%%%%%%%%%%%%%%%%%%%%%%%%%%%%%%%%%%%%%%%%%%%%%%%%%%%%%%%%%%%%%%%%%%%%%%%%%%%%%%%%%%%%%%%%%%%%%%%%%%%%%%%%%%%%%%%%%%
\section{The Consistency Analysis}
\setcounter{equation}{0}

In this section we establish a truncation error estimate for the fully discrete scheme (\ref{numerical}) for the GB equation (\ref{bsq}).  
A finite Fourier projection of the exact solution is taken to the GB equation (\ref{bsq}) and a local truncation error is derived.  Moreover, we perform a higher order consistency analysis in time, through an addition of a correction term, so that the constructed of approximate solution satisfies the numerical scheme with higher order temporal accuracy. This approach avoids a key difficulty associated with the accuracy reduction in time due to the appearance of the second in time temporal derivative. 
%Thirdly, by the triangle inequality, the global error estimate can be obtained.

\subsection{Truncation error analysis for $U_N$} 
Given the domain $\Omega = (0,L)$, the uniform mesh grid $(x_i)$, 
$0 \le i \le 2N$, and the exact solution $u_\exac$,  we denote 
$U_N$ as its projection into ${\cal B}^N$: 
\begin{equation}
U_N (x,t) := {\cal P}_N u_\exac (x,t)  .
\label{consistency-Phi-2}
\end{equation}
The following approximation estimates are clear: 
\begin{eqnarray} 
  &&
  \left\|  U_N - u_\exac \right\|_{L^\infty (0,T^*; H^r) }  
  \le  C  h^m \left\|  u_\exac \right\|_{L^\infty (0,T^*; H^{m+r}) } ,  \quad 
  \mbox{for} \, \, \, r \ge 0 , \label{consistency-Phi-3-1}
\\
  &&
  \left\|  \partial_t^k \left( U_N - u_\exac \right) \right\|_{H^r}  
  \le  C  h^m \left\|  \partial_t^k u_\exac \right\|_{H^{m+r}} ,  \quad 
  \mbox{for} \, \, \, r \ge 0 , \, \, 0 \le k \le 4 ,   \label{consistency-Phi-3-2} 
\end{eqnarray}
in which the second inequality comes from the fact that 
$ \partial_t^k U_N$ is the truncation of $\partial_t^k u_\exac$ 
for any $k \ge 0$, since projection and differentiation commute: 
\begin{equation} 
   \frac{\partial^k}{\partial t^k}  U_N (\x, t) 
 = \frac{\partial^k}{\partial t^k}  {\cal P}_N u_\exac (x,t)  
 = {\cal P}_N  \frac{\partial^k  u_\exac (x,t) }{\partial t^k} 
  . \label{consistency-Phi-3-3}
\end{equation} 
As a direct consequence, the following linear estimates are straightforward: 
\begin{eqnarray}  
  \left\|  \partial_t^2 \left( U_N - u_\exac \right) \right\|_{L^2}  
  &\le&  C  h^m \left\|  \partial_t^2 u_\exac \right\|_{H^m}   ,  
  \label{consistency-Phi-4-1}
\\
  \left\|  \partial_x^2 \left( U_N - u_\exac \right) \right\|_{L^2}  
  &\le&  C  h^m \left\|  u_\exac \right\|_{H^{m+2} }  ,  \quad 
  \left\|  \partial_x^4 \left( U_N - u_\exac \right) \right\|_{L^2}  
  \le  C  h^m \left\|  u_\exac \right\|_{H^{m+4} }  .   \label{consistency-Phi-4-2}
\end{eqnarray} 
On the other hand, a discrete $\nrm{ \cdot }_2$ estimate for these terms are needed in the local truncation derivation. To overcome this difficulty, we observe that 
\begin{equation}  \label{consistency-Phi-4-3}
  \nrm{  \partial_t^2 \left( U_N - u_\exac \right) }_2 
  = \nrm{  {\cal I}_N \left( \partial_t^2 \left( U_N - u_\exac \right)  \right) }_{L^2} 
  \le \nrm{  \partial_t^2 \left( U_N - u_\exac \right)  }_{L^2}  
  + \nrm{  \partial_t^2 \left( {\cal I}_N u_\exac 
  - u_\exac \right)  }_{L^2} , 
\end{equation} 
in which the second step comes from the fact that ${\cal I}_N \partial_t^2 U_N 
= \partial_t^2 U_N$, since $\partial_t^2 U_N \in {\cal B}^N$. The first term has an estimate given by (\ref{consistency-Phi-4-1}), while the second term could be bounded by 
\begin{equation} 
   \nrm{  \left( \partial_t^2 \left( {\cal I}_N u_\exac 
  - u_\exac \right)  \right) }_{L^2} 
  =  \nrm{  {\cal I}_N \left( \partial_t^2 u_\exac \right) 
  - \partial_t^2 u_\exac }_{L^2} 
  \le  C h^m  \nrm{  \partial_t^2 u_\exac }_{H^m} , 
  \label{consistency-Phi-4-4}
\end{equation}
as an application of (\ref{spectral-approximation}). In turn, its combination with (\ref{consistency-Phi-4-3}) and (\ref{consistency-Phi-4-1}) yields 
\begin{equation} 
   \nrm{  \partial_t^2 \left( U_N - u_\exac \right) }_2
  \le  C h^m  \nrm{  \partial_t^2 u_\exac }_{H^m} . 
  \label{consistency-Phi-4-5}
\end{equation}
Using similar arguments, we also arrive at 
\begin{equation}  
  \left\|  \partial_x^2 \left( U_N - u_\exac \right) \right\|_2  
  \le  C  h^m \left\|  u_\exac \right\|_{H^{m+2} }  ,  \quad 
  \left\|  \partial_x^4 \left( U_N - u_\exac \right) \right\|_2 
  \le  C  h^m \left\|  u_\exac \right\|_{H^{m+4} }  .   
  \label{consistency-Phi-4-6}
\end{equation}

For the nonlinear term, we begin with the following expansion: 
\begin{eqnarray} 
  \partial_x^2 \left( u_\exac^p \right) 
  &=&  p  \left(  (p-1) u_\exac^{p-2} (u_\exac)_x^2 + u_\exac^{p-1} (u_\exac)_{xx} \right) ,   \quad 
  \mbox{which in turn gives}   \nonumber 
\\
  \partial_x^2 \left( u_\exac^p - (U_N)^p  \right)  
  &=&  p  \Bigl(  (p-1) U_N^{p-2} ( u_\exac + U_N )_x  ( u_\exac - U_N )_x   \nonumber 
\\
  && 
  + (p-1)  ( u_\exac - U_N )  (u_\exac)_x^2 \sum_{k=0}^{p-3} u_\exac^k U_N^{p-3-k}   \nonumber 
\\
  && 
  +  U_N^{p-1}  ( u_\exac - U_N )_{xx} 
  + ( u_\exac - U_N )  (u_\exac)_{xx} \sum_{k=0}^{p-2} u_\exac^k U_N^{p-2-k}  \Bigr) . 
   \label{consistency-Phi-5-1}
\end{eqnarray} 
Subsequently, its combination with (\ref{consistency-Phi-3-1}) implies that 
\begin{eqnarray} 
  &&
   \left\|  \partial_x^2 \left( u_\exac^p - (U_N)^p  \right)   \right\|_{L^2}  
  \nonumber 
\\
  &\le& 
  C  \Bigl( \nrm{ U_N }_{L^\infty}^{p-2}    
    \cdot \left\| u_\exac  + U_N \right\|_{W^{1,\infty} } 
    \cdot  \left\|  u_\exac - U_N   \right\|_{H^1}  
     +  \left\|  U_N   \right\|_{L^\infty}^{p-1}  \cdot 
   \left\|   u_\exac - U_N   \right\|_{H^2}   \nonumber 
\\
  &&
    + \left\|  u_\exac - U_N   \right\|_{L^\infty}  
  \cdot \left( \nrm{ u_\exac }_{L^\infty}^{p-2} + \nrm{ U_N }_{L^\infty}^{p-2} \right)
   \cdot  \left( \left\|  u_\exac  \right\|_{H^2} + \nrm{ u_\exac }_{W^{1,4} }^2 \right)  \Bigr)  \nonumber 
\\
  &\le& 
  C  \Bigl(   \nrm{ U_N }_{H^1}^{p-2}  
  \cdot \left\| u_\exac  + U_N \right\|_{H^2 } 
    \cdot  \left\|  u_\exac - U_N   \right\|_{H^1}  
   +  \left\|  U_N   \right\|_{H^1}^{p-1}  \cdot 
   \left\|   u_\exac - U_N   \right\|_{H^2}   \nonumber
\\
  && 
  + \left\|  u_\exac - U_N   \right\|_{H^1}  
  \cdot \left( \nrm{ u_\exac }_{H^1}^{p-2} + \nrm{ U_N }_{H^1}^{p-2}  \right) 
   \cdot  \left( \left\|  u_\exac  \right\|_{H^2} + \nrm{ u_\exac }_{H^2 }^2 \right)  \Bigr) 
  \nonumber 
\\
  &\le& 
  C  \left(   \left\| u_\exac  \right\|_{H^2}^p +  \left\|  U_N \right\|_{H^2 }^p  \right) 
    \cdot  \left\|  u_\exac - U_N   \right\|_{H^2}    \nonumber 
\\
  &\le& 
   C   \left\| u_\exac  \right\|_{H^2}^p  \cdot  \left\|  u_\exac - U_N   \right\|_{H^2}  
   \le  C  h^m \left\| u_\exac   \right\|_{H^2}^p 
  \cdot  \left\|  u_\exac \right\|_{H^{m+2} }  ,  
   \label{consistency-Phi-5-2}
\end{eqnarray} 
in which an 1-D Sobolev embedding was used in the second step. 

The following interpolation error estimates can be derived in 
a similar way, based on \qref{spectral-approximation}:  
\begin{eqnarray} 
  \nrm{  \partial_x^2 ( u_\exac^p ) 
   - {\cal I}_N \left(  \partial_x^2 ( u_\exac^p ) \right) }_{L^2}  
  &\le&  C h^m  \nrm{  \partial_x^2 ( u_\exac^p )  }_{H^m} 
  \le  C  h^m \nrm{ u_\exac  }_{H^2}^p 
  \cdot  \nrm{  u_\exac } _{H^{m+2} }  ,  \label{consistency-Phi-5-3}
\\
   \nrm{  \partial_x^2 ( U_N^p ) 
   - {\cal I}_N \left(  \partial_x^2 ( U_N^p ) \right) }_{L^2}  
  &\le&  C h^m  \nrm{  \partial_x^2 ( U_N^p )  }_{H^m} 
  \le  C  h^m \nrm{ u_\exac  }_{H^2}^p 
  \cdot  \nrm{  u_\exac } _{H^{m+2} }  .   \label{consistency-Phi-5-4}
\end{eqnarray} 
In turn, a combination of \qref{consistency-Phi-5-2}-\qref{consistency-Phi-5-4} 
implies the following estimate for the nonlinear term 
\begin{eqnarray} 
   \nrm{  \partial_x^2 \left( u_\exac^p - (U_N)^p  \right)  }_2 
  &=&  \nrm{  {\cal I}_N  \left( \partial_x^2 \left( u_\exac^p 
  - (U_N)^p  \right)  \right) }_{L^2}  \nonumber 
\\
  &\le&    
   \nrm{  \partial_x^2 \left( u_\exac^p - (U_N)^p  \right)   } _{L^2}  
   + \nrm{  \partial_x^2 ( u_\exac^p ) 
   - {\cal I}_N \left(  \partial_x^2 ( u_\exac^p ) \right) }_{L^2}    \nonumber 
\\
  && 
  + \nrm{  \partial_x^2 ( U_N^p ) 
   - {\cal I}_N \left(  \partial_x^2 ( U_N^p ) \right) }_{L^2}  
  \le  C  h^m \nrm{ u_\exac  }_{H^2}^p 
  \cdot  \nrm{  u_\exac } _{H^{m+2} }  .   \label{consistency-Phi-5-5}
\end{eqnarray} 

   By observing \qref{consistency-Phi-4-5}, \qref{consistency-Phi-4-6}, \qref{consistency-Phi-5-5}, we conclude that $U_N$ satisfies the original GB equation \qref{bsq} up to an $O (h^m)$ (spectrally accurate) truncation error: 
\begin{eqnarray} 
    \partial_t^2 U_N = - \partial_x^4 U_N + \partial_x^2 U_N 
  + \partial_x^2 (U_N^p) + \tau_0 ,  \quad \mbox{with} \, \, \, 
  \nrm{ \tau_0 }_2 \le C  h^m \left( \nrm{ u_\exac  }_{H^2}^p + 1 \right) 
  \cdot  \nrm{  u_\exac } _{H^{m+4} }  .   \label{consistency-Phi-6}
\end{eqnarray}

Moreover, we define the following profile, a second order (in time) 
approximation to $\partial_t u_\exac$: 
\begin{equation}\label{psi-approx}
\Psi_N (x,t) := \partial_t U_N (x,t) - \frac{\Delta t^2}{12} \partial_t^3 U_N (x,t) .
\end{equation}
For any function $G=G(x,t)$, given $n>0$ , we define $G^n(x) := G(x, n \dt)$.

%Before giving our result, we cite the following propositions.

\subsection{Truncation error analysis in time} 

  For simplicity of presentation, we assume $T = K \dt$ with an integer $K$. 
The following two preliminary estimates are excerpted from a recent work 
\cite{Lowengrub}, which will be useful  in later consistency analysis.

\begin{pro}\cite{Lowengrub}  
\label{Prop-A.0}
For $f \in H^3 (0,T)$, we have
\begin{equation}\label{est-1d-time-1}
\nrm{\tau^t f }_{\ell^2 (0,T)}  \le C \dt^m \nrm{f}_{H^{m+1} (0,T)}  ,  \quad
\mbox{with} \quad \tau^t f ^n = \frac{f^{n+1} - f^n}{\dt} - f' (t^{n+1/2}) ,
\end{equation}
for $0 \le m \le 2$, where $C$ only depends on $T$, $\nrm{\ \cdot \ }_{\ell^2 (0,T)}$ is a discrete $L^2$ norm (in time) given by $\nrm{g}_{\ell^2 (0,T)} = \sqrt{\Delta t \sum_{n=0}^{K-1} \left(g^n\right)^2}$.
\end{pro}

\begin{pro}\cite{Lowengrub}
\label{Prop-A.1}
For $f \in H^2 (0,T)$, we have
	\begin{eqnarray}
  &&
  \nrm{D_{t/2}^2 f }_{\ell^2 (0,T)}  :=
  \left( \dt \sum_{n=0}^{K-1} \left( D_{t/2}^2 f^{n+1/2} \right)^2 \right)^{\frac12}
 \le C \nrm{f}_{H^2 (0,T)}  ,  \label{est-1d-time-2-1}
%\\
%  &&
%  \mbox{with}  \quad
% D_{t/2}^2 f^{n+1/2}  = \frac{4 \left( u^{n+1} - 2 f (\  \cdot\ , t^{n+1/2} ) + f^n \right)}{\Delta t^2} ,\nonumber
\\
  &&
  \nrm{D_t^2 f }_{\ell^2 (0,T)}  :=
  \left( \dt \sum_{n=0}^{K-1} \left( D_t^2 f^n \right)^2 \right)^{\frac12}
 \le C \nrm{f}_{H^2 (0,T)}  ,  \label{est-1d-time-2-2}
\\
  &&\mbox{with}  \quad D_{t/2}^2 f^{n+1/2}  = \frac{4 \left( f^{n+1} - 2 f (\  \cdot\ , t^{n+1/2} ) + f^n \right)}{\Delta t^2} , \quad D_t^2 f^n  = \frac{f^{n+1} - 2 f^n + f^{n-1} }{\Delta t^2} ,  \nonumber
\end{eqnarray}
where $C$ only depends on $T$.
\end{pro}

%\begin{pro}\cite{Lowengrub}
%\label{Prop-A.2}
%If $u \in {\cal B}^{N/2}$ has a regularity $u \in H^{8}_{per} (\Omega)$, we have
%\begin{equation}\label{est-2d-1}
%\nrm{\Delta^n u - \Delta_h^n u}_{L_h^2 (\Omega)}
%\le C h^2 \nrm{u}_{H^{2 + 2n} (\Omega)} ,  \quad
%\mbox{for}  \quad n= 1, 2, 3,
%\end{equation}
%where $C$ only depends on $L_0$ and $\nrm{g}_{L_h^2 (\Omega)} = \sqrt{h^2 \sum_{i,j=0}^{N-1} g_{i,j}^2}$.
%\end{pro}

  The following theorem is the desired consistency result. To simplify the 
presentation below, for the constructed solution $(U_N, \psi_N)$, we define its vector grid function 
$(U^n, \Psi^n) = {\cal I} (U_N, \psi_N)$ as its interpolation: 
$U^n_i = U_N^n (x_i, t^n)$, $\Psi^n_i = \Psi_N^n (x_i, t^n)$.  

\begin{thm}\label{thm1}
Suppose the unique periodic solution for equation (\ref{bsq}) 
%is given by 
satisfies the following regularity assumption
\begin{equation}\label{Bos-regularity}
    u_\exac \in H^4 (0,T; L^2) \cap L^\infty (0,T; H^{m+4}) \cap H^2 (0,T; H^4) . 
\end{equation}
%for $T< \infty$. 
Set $(U_N, \Psi_N)$ as the approximation solution constructed 
by \qref{consistency-Phi-2}, \qref{psi-approx} and let $(U, \Psi)$ as its discrete interpolation. Then we have 
\begin{equation}\label{numerical2}
\left\{\begin{array}{rl}
  \dspace 
\frac{\Psi^{n+1}-\Psi^{n}}{\Delta t}
&=-D_N^4 (\frac{U^{n+1}+U^n}{2}) + D_N^2(\frac{U^{n+1}+U^n}{2})\\
  \dspace 
& +D_N^2 \left(\frac{3}{2}(U^n)^p-\frac{1}{2}(U^{n-1})^p \right)+\tau_1^n,\\
  \dspace 
\frac{U^{n+1}-U^n}{\Delta t}
&=\frac{\Psi^{n+1} + \Psi^{n}}{2}+\Delta t \tau_2^n,
\end{array}\right.
\end{equation}
where $\tau_i^k$  satisfies 
	\begin{equation}
\nrm{\tau_i}_{\ell^2\left(0,T; \ell^2 \right)} := \left( \dt \sum_{k=0}^{K} \nrm{\tau_i^{k}}^2_2 \right)^{\frac12} \le M \left( \dt^2 + h^m \right) ,  
  \quad i = 1, 2 , 
	\label{discrete-norm-truncation}
	\end{equation}
in which $M$ only depends on the regularity of the exact solution $u_\exac$. 
%	\begin{eqnarray}
%M  &\le& C \Bigl(\nrm{U}_{H^4 \left(0,T; L^2 \left( \Omega \right)\right)}
%+ \nrm{U}_{W^{2,\infty} \left(0,T; H^2 \left( \Omega \right)\right)}
%+ \nrm{U}_{H^{2} \left(0,T; H^{4}\left(\Omega\right)\right)}	
%	\nonumber
%\\
%  &&
%+ \nrm{U}_{L^\infty  \left(0,T; H^{m+4}\left(\Omega\right)\right)}
%+ \nrm{U}_{L^\infty \left(0,T; H^{m+2}\left(\Omega\right)\right)}   \cdot
%  \left( 1 + \nrm{U}_{L^\infty \left(0,T; H^{m+2}\left(\Omega\right)\right)} \right)
%  \nonumber
%\\
%  &&
%    + \nrm{U}_{H^2 \left(0,T; H^{2}\left(\Omega\right)\right)}   \cdot
%  \left( 1 + \nrm{U}_{H^2 \left(0,T; H^{2}\left(\Omega\right)\right)}  \right) \Bigr) \ .
%	\label{truncation-error}
%	\end{eqnarray}
\end{thm}

\begin{proof}
%We give the local truncation error estimate of (\ref{Bos-regularity}) for vertex-centered grid functions. The following results will
%be used to establish (\ref{Bos-regularity}).

%\begin{pro}\cite{Lowengrub}
%For the projection $U_N$, we have the following approximation estimate
%\begin{equation}\label{consistency-Phi-4}
%\nrm{U_N-U}_{L(0,T;H^r)} \le  C  h^m \nrm{ U}_{L^\infty(0,T;H^{m+r})} \ .
%\end{equation}
%\end{pro}

We define the following notation:
\begin{equation}\label{notation1}
     \begin{array}{rclrclrcl}
     F_0^{n+1/2}&=&\frac{U^{n+1}-U^{n}}{\dt}    &,&  \\ % \quad F_{0e}^{n+1/2}&=&(\partial_t U_N)(\cdot,t^{n+1/2}),\\
    F_1^{n+1/2}&=&\frac{\Psi^{n+1}-\Psi^{n}}{\Delta t}&,&\quad F_{1e}^{n+1/2}&=&(\partial_t^2 U_N)(\cdot,t^{n+1/2}),\\
    F_2^{n+1/2}&=&D_N^4 U^{n+1/2}&,&\quad F_{2e}^{n+1/2}&=&(\partial_x^4 U_N)(\cdot,t^{n+1/2}),\\
    F_3^{n+1/2}&=&D_N^2 U^{n+1/2}&,&\quad F_{2e}^{n+1/2}&=&(\partial_x^2 U_N)(\cdot,t^{n+1/2}),\\
    F_4^{n+1/2}&=&D_N^2 (\frac{3}{2}(U^p)^{n}-\frac{1}{2}(U^p)^{n-1})&,&\quad F_{4e}^{n+1/2}&=&(\partial_x^2 U_N^2)(\cdot,t^{n+1/2}),\\
    F_5^{n+1/2}&=&\frac{\Psi^{n+1}+\Psi^{n}}{2}  &.&  \\
     \end{array}
\end{equation}
Note that the quantities on the left side are defined on the numerical grid (in space) point-wise, while the ones on the right hand side are continuous functions. 

%Moreover, the corresponding exact solution value are defined by
%\begin{equation}\label{notation2}
%     \begin{array}{rclrclrcl}
%     F_{0en}^{n+1/2}&=&\partial_t U(\cdot,t^{n+1/2})&,&F_{1en}^{n+1/2}&=&\partial_t^2 U(\cdot,t^{n+1/2})\\
%     F_{2en}^{n+1/2}&=&\partial_x^4 U(\cdot,t^{n+1/2})&,&F_{3en}^{n+1/2}&=&\partial_x^2 U(\cdot,t^{n+1/2})\\
%     \quad F_{4en}^{n+1/2}&=&(\partial_x^2 U^2)(\cdot,t^{n+1/2}).\\
%     \end{array}
%\end{equation}
%Note that all these quantities are defined on the numerical grid (in space) point-wise.

  To begin with, we look at 
%first order time derivative terms, $F_0$ and $F_{0e}$. A direct application of Proposition~\ref{Prop-A.0}
%indicates that (by taking $m=2$):
%       \begin{eqnarray}
%\nrm{ F_0 - F_{0e} }_{\ell^2 (0,T)}  \le C \Delta t^2 \nrm{ U_N}_{H^{3} (0,T)} ,
%	\label{truncation-est-2nd-2-1}
%        \end{eqnarray}
%for each fixed grid point. In turn, its combination with a discrete summation in $\Omega$ shows that 
%\begin{eqnarray}
%\nrm{ F_0 - F_{0e} }_{\ell^2 \left(0,T; \ell^2 \right)} 
%  \le C \Delta t^2 \nrm{ U_N}_{H^{3} (0,T; L^2)} 
%  \le C \Delta t^2 \nrm{ U }_{H^{3} (0,T; L^2)} , 
%	\label{truncation-est-2nd-2-3}
%\end{eqnarray}
%due to the fact that $U_N \in {\cal B}^N$, and \qref{consistency-Phi-3-2} 
%was used in the second step. 
the second order time derivative terms, $F_1$ and $F_{1e}$. 
%a similar idea can be applied. 
From the definition (\ref{psi-approx}), we get
 \begin{equation} \label{truncation-est-f1}
    F_1^{n+1/2} =\frac{\partial_t u_N^{n+1}-\partial_t u_N^{n}}{\Delta t}-\frac{\Delta t^2}{12}\frac{\partial_t^3 u_N^{n+1}-\partial_t^3 u_N^{n}}{\Delta t} :=F_{11}^{n+1/2}-\frac{\Delta t^2}{12}F_{12}^{n+1/2},
\end{equation}
at a point-wise level, where $F_{11}$ and $F_{12}$ are the finite difference (in time) approximation to $\partial_t^2 U_N$, $\partial_t^4 U_N$, respectively.
We define $F_{11e}$ and $F_{12e}$ in a similar way as (\ref{notation1}), i.e.
\begin{equation}\label{notation3}
    F_{11e}^{n+1/2} = \partial_t^2 U_N(\cdot,t^{n+1/2}) ,  \quad  
    F_{12e}^{n+1/2} = \partial_t^4 U_N(\cdot,t^{n+1/2}) . 
\end{equation}
The following estimates
can be derived by using Proposition~\ref{Prop-A.0} (with $m=2$ and $m=0$):
\begin{eqnarray}
    \nrm{ F_{11} - F_{11e} }_{\ell^2 (0,T)}  
  \le C \Delta t^2 \nrm{ U_N}_{H^{4} (0,T)} ,  \quad 
    \nrm{ F_{12} - F_{12e} }_{\ell^2 (0,T)}  
   \le C \nrm{ U_N}_{H^{4} (0,T)} ,    \label{truncation-est-2nd-1-2}
\end{eqnarray}
for each fixed grid point. This in turn yields 
\begin{eqnarray}
  \nrm{  F_{1} - F_{1e}  }_{\ell^2 (0,T)}
  \le C \Delta t^2 \left\|  U_N  \right\|_{H^4(0,T)} .
  \label{truncation-est-2nd}
\end{eqnarray}
In turn, an application of discrete summation in $\Omega$ leads to
\begin{eqnarray}
\nrm{ F_1 - {\cal I} ( F_{1e} ) }_{\ell^2 \left(0,T; \ell^2 \right)} 
   \le C  \Delta t^2  \nrm{ U_N }_{H^4 (0,T; L^2)} 
  \le C  \Delta t^2  \nrm{ u_\exac }_{H^4 (0,T; L^2)} ,
	\label{truncation-est-2nd-1-5}
\end{eqnarray}
due to the fact that $U_N \in {\cal B}^N$, and \qref{consistency-Phi-3-2} 
was used in the second step. 

For the terms $F_2$ and $F_{2e}$, we start from the following observation
(recall that $U_N^{k+1/2} = \frac{U_N^{k+1} + U_N^k}{2}$) 

\begin{equation}\label{truncation-est-2nd-4-1}
  \nrm{F_2^{n+1/2} -  {\cal I} \left( \partial_x^4 U_N^{n+1/2} \right) }_2 \equiv 0 ,  \quad 
  \mbox{since  $U_N^{n+1/2}  \in {\cal B}^N$ } .  
\end{equation}
Meanwhile, a comparison between $U_N^{n+1/2}$ and $ U_N ( \cdot\ , t^{n+1/2} )$  shows that
\begin{equation}\label{truncation-est-2nd-4-2}
  U_N^{n+1/2} - U_N (\  \cdot\ , t^{n+1/2} )
  = \frac18 \Delta t^2 D_{t/2}^2 U_N^{n+1/2} .
\end{equation}
Meanwhile, an application of Prop.~\ref{Prop-A.1} gives
\begin{equation}\label{truncation-est-2nd-4-3}
    \nrm{D_{t/2}^2 \partial_x^4 U_N}_{\ell^2 (0,T)}\leq C \nrm{\partial_x^4 U_N}_{H^2 (0,T)},
\end{equation}
at each fixed grid point. As a result, we get
\begin{equation}\label{truncation-est-2nd-4-4}
\nrm{ F_2 - {\cal I} \left( F_{2e} \right) }_{\ell^2 \left(0,T; \ell^2 \right)}
\leq C \Delta t^2  \nrm{  u_\exac  }_{H^2 (0,T; H^4)}  .
\end{equation}

The terms $F_3$ and $F_{3e}$ can be analyzed in the same way. We have
 \begin{equation}\label{truncation-est-2nd-4-7}
  \nrm{ F_3 - {\cal I} \left( F_{3e} \right) }_{\ell^2 (0,T; \ell^2 )}
  \leq C \Delta t^2 \nrm{  u_\exac  }_{H^2 (0,T; H^2)}  .
 \end{equation}
	
For the nonlinear terms $F_4$ and $F_{4e}$, we begin with the following estimate 
\begin{equation}\label{truncation-est-2nd-4-8}
\begin{split}
 & \nrm{ F_4^{n+1/2} - {\cal I} \left( \partial_x^2\left( \frac{3}{2}(U_N^p )^n-\frac{1}{2}(U_N^p )^{n-1} \right)  \right) }_2
 \leq C h^m \nrm{ \frac{3}{2}(U_N^p )^n-\frac{1}{2}(U_N^p)^{n-1}}_{H^{m+2}} \\
& \quad \leq C h^m \left(\nrm{ U_N^n }_{H^{m+2}}^p + \nrm{ U_N^{n-1}}_{H^{m+2}}^p\right) 
  \leq C h^m \nrm{ U_N }_{L^\infty{(0,T;H^{m+2}})}^p , 
\end{split}
\end{equation}
with the first step based on the fact that $\frac{3}{2}(U_N^p )^n-\frac{1}{2}(U_N^p)^{n-1} \in {\cal B}^{p N}$. Meanwhile, the following observation
\begin{eqnarray}
   \frac{3}{2}(U_N^p )^n-\frac{1}{2}(U_N^p)^{n-1}-U_N^p (\cdot,t^{n+1/2})  %\nonumber\\
%&& =\frac{(U_N^2)^{n}-2U_N^2(\cdot,t^{n+1/2})+(U_N^2)^{n+1}}{2}-\frac{(U_N^2)^{n+1}-2(U_N^2)^{n}+(U_N^2)^{n-1}}{2}\nonumber\\
  =\frac{1}{8}\Delta t^2 D_{t/2}^2(U_N^p)-\frac{1}{2}\Delta t^2 D_{t}^2(U_N^p) 
    \label{truncation-est-2nd-3-2}
\end{eqnarray}
indicates that  
\begin{equation}\label{truncation-est-2nd-4-9}
\begin{split}
&\nrm{ {\cal I} \left(  \partial_x^2\left( \frac{3}{2}(U_N^p )^n-\frac{1}{2}(U_N^p )^{n-1}\right) - F_{4e}^{n+1/2}  \right) }_2 \\
%&\qquad =  \nrm{\partial_x^2\left( \frac{3}{2}(U_N^p )^n-\frac{1}{2}(U_N^p)^{n-1}-U_N^p(\cdot,t^{n+1/2})\right) }_{L_s^2(\Omega)}\\
&\qquad = \nrm{ {\cal I}  \left( \partial_x^2\left( \frac{1}{8}\Delta t^2 D_{t/2}^2(U_N^p)-\frac{1}{2}\Delta t^2 D_{t}^2(U_N^p)\right) \right) }_2  \\
&\qquad \leq \frac18 \Delta t^2\nrm{D_{t/2}^2(U_N^p) }_{H^{2+\eta}}  + \frac12 \Delta t^2\nrm{ D_{t}^2(U_N^p) }_{H^{2+\eta} } , \quad \eta > \frac12 , 
\end{split}
\end{equation}
with the last step coming from \qref{spectral-approximation}. On ther other hand, applications of Prop.~\ref{Prop-A.0}, Prop.~\ref{Prop-A.1} imply that 
\begin{eqnarray} 
  \nrm{D_{t/2}^2 (U_N^p) }_{\ell^2 (0,T; H^3)}
 \le C \nrm{U_N^p}_{H^2 (0,T; H^3)}  ,  \quad  
  \nrm{D_{t}^2 (U_N^p) }_{\ell^2 (0,T; H^3)}
 \le C \nrm{U_N^p}_{H^2 (0,T; H^3)} . \label{truncation-est-2nd-3-5}
\end{eqnarray}
Note that an $H^2$ estimate (in time) is involved with a nonlinear term $U_N^p$.
A detailed expansion in its first and second order time derivatives shows that
\begin{eqnarray}
  \partial_t (U_N^p) = p U_N^{p-1} \partial_t U_N ,  \quad
  \partial_t^2 (U_N^p) = p \left(  U_N^{p-1} \partial_t^2 U_N
  + (p-1) U_N^{p-2} ( \partial_t U_N )^2 \right) ,  \label{truncation-est-2nd-3-7}
\end{eqnarray}
which in turn leads to
\begin{eqnarray}
  \nrm{U_N^p}_{H^2 (0,T)}
  &\le&
  C \left(  \nrm{U_N}_{L^\infty (0,T)}^{p-1}  \cdot  \nrm{U_N}_{H^2 (0,T)}
  +  \nrm{U_N}_{L^\infty (0,T)}^{p-2} \cdot \nrm{U_N}_{W^{1,4} (0,T)}^2 \right)  \nonumber
\\
  &\le&
  C  \nrm{U_N}_{H^2 (0,T)}^p ,   \label{truncation-est-2nd-3-8}
\end{eqnarray}
at each fixed grid point, with an 1-D Sobolev embedding applied at the last step.
Going back to (\ref{truncation-est-2nd-3-5}) gives
\begin{eqnarray}
  \nrm{D_{t/2}^2 (U_N^p) }_{\ell^2 (0,T; H^3)} 
 \le C \nrm{U_N}_{H^2 (0,T; H^3)}^p  ,  \quad 
   \nrm{D_{t}^2 (U_N^p) }_{\ell^2 (0,T; H^3)}
 \leq C \nrm{U_N}_{H^2 (0,T; H^3)}^p . 
 \label{truncation-est-2nd-3-9}
\end{eqnarray}
A combination of (\ref{truncation-est-2nd-4-9}), (\ref{truncation-est-2nd-3-9}) and (\ref{truncation-est-2nd-4-8}) leads to the consistency estimate of the nonlinear term
\begin{eqnarray}
\nrm{ F_4 - {\cal I} \left( F_{4e} \right) }_{\ell^2 \left(0,T; \ell^2 \right)}
  \le C (\Delta t^2 + h^m) \left(\nrm{ u_\exac }_{H^2 (0,T; H^3)}^p + \nrm{ u_\exac }_{L^\infty{(0,T;H^{m+2}})}^p \right) .
	\label{truncation-est-2nd-4-13}
\end{eqnarray}

Therefore, the local truncation error estimate for $\tau_1$ is obtained by combining (\ref{truncation-est-2nd-1-5}), (\ref{truncation-est-2nd-4-4}), (\ref{truncation-est-2nd-4-7}) and (\ref{truncation-est-2nd-4-13}), 
combined with the consistency estimate \qref{consistency-Phi-6} for $U_N$. 
Obviously, constant $M$ only dependent on the exact solution $u_\exac$. 

The estimate for $\tau_2$ is very similar.
We denote the following quantity
\begin{eqnarray}
  F_{5e}^{n+1/2} = \left( \partial_t U_N + \frac{\dt^2}{24} \partial_t^3  U_N\right)
  (\  \cdot\ , t^{n+1/2} )  .  \label{truncation-est-2nd-7-1}
\end{eqnarray}
A detailed Taylor formula in time gives the following estimate:
\begin{eqnarray}
  &&
  F_0^{n+1/2} - {\cal I} \left( F_{5e}^{n+1/2} \right) = \tau_{21}^{n+1/2} , \quad 
  \mbox{with}  \, \, \, \nonumber
\\
  &&
  \left\|  \tau_{21}  \right\|_{\ell^2 (0,T)}
  \le C \Delta t^3  \left\|  U_N  \right\|_{H^4 (0,T)}
   \le C \Delta t^3  \left\|  u_\exac  \right\|_{H^4 (0,T)}  ,  \label{truncation-est-2nd-7-2}
\end{eqnarray}
at each fixed grid point. Meanwhile, from the definition of (\ref{psi-approx}), it is clear that $F_5$ has the following decomposition:
\begin{eqnarray}
   F_5^{n+1/2}
   &=& \frac{ \Psi_N^{n+1} + \Psi_N^n}{2}
   = \frac{ \partial_t U_N^{n+1} + \partial_t U_N^n}{2}
   - \frac{\dt^2}{12}  \cdot   \frac{ \partial_t^3 U_N^{n+1} + \partial_t^3 U_N^n}{2}
   \nonumber
\\
  &:=& F_{51}^{n+1/2} + F_{52}^{n+1/2} , 
   \label{truncation-est-2nd-7-3}
\end{eqnarray}
at a point-wise level. To facilitate the analysis below, we define two more quantities:
\begin{eqnarray}
%  &&
  F_{51e}^{n+1/2} = \left( \partial_t U_N + \frac{\Delta t^2}{8} \partial_t^3  U_N \right)  (\  \cdot\ , t^{n+1/2} ) ,  \nonumber
%\\
%  &&
  \quad  F_{52e}^{n+1/2} = - \frac{\Delta t^2}{12} \partial_t^3  U_N  (\  \cdot\ , t^{n+1/2} )
   .  \label{truncation-est-2nd-7-4}
\end{eqnarray}
A detailed Taylor formula in time gives the following estimate:
\begin{eqnarray}
  &&
  F_{51}^{n+1/2} - {\cal I} \left( F_{51e}^{n+1/2} \right) = \tau_{22}^{n+1/2} ,   \quad
  F_{52}^{n+1/2} - {\cal I} \left( F_{52e}^{n+1/2} \right) = \tau_{23}^{n+1/2}  ,  \quad
  \mbox{with}  \, \, \, \nonumber
\\
  &&
  \left\|  \tau_{22}  \right\|_{\ell^2 (0,T)}
  \le C \Delta t^3  \left\|  U_N  \right\|_{H^4 (0,T)}
   \le C \Delta t^3  \left\|  u_\exac  \right\|_{H^4 (0,T)}  ,  \label{truncation-est-2nd-7-5}
\\
  &&
  \left\|  \tau_{23}  \right\|_{\ell^2 (0,T)}
  \le C \Delta t^3  \left\|  U_N  \right\|_{H^4 (0,T)}
   \le C \Delta t^3  \left\|  u_\exac  \right\|_{H^4 (0,T)}  ,  \label{truncation-est-2nd-7-6}
\end{eqnarray}
at each fixed grid point. Consequently, a combination of
(\ref{truncation-est-2nd-7-2})-(\ref{truncation-est-2nd-7-6}) shows that
\begin{eqnarray}
  F_0^{n+1/2} - F_{5}^{n+1/2} = \tau_2^{n+1/2} , \quad 
  \mbox{with}  \, \, \,
  \left\|  \tau_{2}  \right\|_{\ell^2 (0,T)}
   \le C \Delta t^3  \left\|  u_\exac  \right\|_{H^4 (0,T)}  .  \label{truncation-est-2nd-7-7}
\end{eqnarray}
This in turn implies that
\begin{eqnarray}
  \left\| F_0 - F_{5} \right\|_{\ell^2 (0,T; \ell^2)}
  \le  C \Delta t^3  \left\|  u_\exac  \right\|_{H^4 (0,T; L^2)}  . \label{truncation-est-2nd-7-8}
\end{eqnarray}
Consequently, a discrete summation in $\Omega$ gives the second estimate in 
\qref{discrete-norm-truncation} (for $i=2$), in which the constant $M$ only dependent on the exact solution. 
The consistency analysis is thus completed.
\end{proof}

%%%%%%%%%%%%%%%%%%%%%%%%%%%%%%%%%%%%%%%%%%%%%%%%%%%%%%%%%%%%%%%%%%%%%%%%%%%%%%%%%%%%%%%%%%%%%%%%%%%%%%%%%%%%%%%%%%%%%%%%%%%%%%%
\section{The Stability and Convergence Analysis }
\setcounter{equation}{0}

  Note that the numerical solution $(u, \psi)$ of \qref{numerical} is a vector function evaluated at discrete 
grid points. Before the convergence statement of the numerical scheme,  
its continuous extension in space is introduced, defined by 
$u_{\dt,h}^k = u_N^k$,  $\psi_{\dt,h}^k = \psi_N^k$, 
in which $u_N^k , \psi_N^k \in {\cal B}^N, \forall k$, are the continuous 
version of the discrete grid functions $u^k$, $\psi^k$, 
with the interpolation formula given by (\ref{spectral-1}). 

The point-wise numerical error grid function is given by 
\begin{equation} 
  \tilde{u}_i^n = U_i^n - u_i^n  ,  \quad 
  \tilde{\psi}_i^n = \Psi_i^n - \psi_i^n  ,    \label{error function} 
\end{equation} 

To facilitate the presentation below, we denote 
$( \tilde{u}_N^n , \tilde{\psi}_N^n) \in {\cal B}^N$  
as the continuous version of the numerical solution $\tilde{u}^n$ 
and $\tilde{\psi}^n$, respectively, with the interpolation formula 
given by (\ref{spectral-1}). 

The following preliminary estimate will be used in later analysis. 
For simplicity, we assume the initial value for $u_t$ for the GB 
equation \qref{bsq} is given by $v^0 (x) = u_t (x, t=0) \equiv 0$. 
The general case can be analyzed in the same manner, 
with more details involved.

\begin{lem} \label{lemma:prelim est}
At any time step $t^k$, $k \ge 0$, we have 
\begin{eqnarray} 
   \| \tilde{u}_N^k \|_{H^2} \leq  
   C \left(  \| D_N^2 \tilde{u}^k \|_2 + h^m \right) , 
    \label{H2 est-1}  
\end{eqnarray} 
\end{lem} 

\begin{proof} 
  First, we recall that the exact solution to the GB equation \qref{bsq} is mass conservative, provided that $v^0 (x) = u_t (x, t=0) \equiv 0$: 
\begin{equation} 
  \int_\Omega u_e (\cdot , t)  \, d x \equiv \int_\Omega u_e (\cdot , 0)  \, d x ,  \quad \mbox{with}  \quad \forall t > 0 .  
  \label{mass conserv-1}
\end{equation} 
Since $U_N$ is the projection of $u_\exac$ into ${\cal B}^N$, as given by 
\qref{consistency-Phi-2}, we conclude that 
\begin{equation} 
   \int_\Omega U_N (\cdot , t)  \, d x
  = \int_\Omega u_e (\cdot , t)  \, d x \equiv \int_\Omega u_e (\cdot , 0)  \, d x 
  = \int_\Omega U_N (\cdot , 0)  \, d x ,  \quad \mbox{with}  \quad \forall t > 0 .  
  \label{mass conserv-2}
\end{equation} 
On the other hand, the numerical scheme (\ref{numerical}) is mass 
conservative at the discrete level, provided that $\psi^0 \equiv 0$: 
\begin{eqnarray}
   \overline{u^k} := h \sum_{i=0}^{N-1} u^k_{i}  
  \equiv \overline{u^0} = \bar{C}_0 .   
  \label{mass conserv-discrete-1}
\end{eqnarray} 
Meanwhile, for $U_N^k \in {\cal B}^N$, for any $k \ge 0$, we observe that 
\begin{eqnarray}
   \overline{U^k} = \int_\Omega U_N (\cdot , t^k)  \, d x
   \equiv \int_\Omega U_N (\cdot , 0) = \overline{U^0} .    
  \label{mass conserv-discrete-2}
\end{eqnarray} 
As a result, we arrive at an $O (h^m)$ order average for the numerical error 
function at each time step: 
\begin{eqnarray}
   \overline{\tilde{u}^k } = \overline{U^k - u^k} 
  = \overline{U^k}  - \overline{u^k}  
  = \overline{U^0}  - \overline{u^0} 
  = O (h^m)  ,  \quad \forall k \ge 0 ,
  \label{error-mean-1}
\end{eqnarray} 
which comes from the error associated with the projection. 
This is equivalent to 
\begin{eqnarray}
   \int_\Omega \tilde{u}_N^k   \, d x 
  = \overline{\tilde{u}^k }  = O (h^m)  ,  \quad \forall k \ge 0 ,
  \label{error-mean-2}
\end{eqnarray} 
with the first step based on the fact that $\tilde{u}_N^k \in {\cal B}^N$. 
As an application of elliptic regularity, we arrive at 
\begin{equation} 
   \| \tilde{u}_N^k \|_{H^2}  
  \le  C  \left(  \nrm{ \partial_x^2 \tilde{u}_N^k }_{L^2} + 
    \int_\Omega \tilde{u}_N^k   \, d x  \right) 
  \le  C  \left(  \nrm{ D_N^2 \tilde{u}^k }_2 +  h^m  \right) , 
\end{equation} 
in which the fact that $\tilde{u}_N^k \in {\cal B}^N$ was used in the last step. 
This finishes the proof of Lemma \ref{lemma:prelim est}. 
\end{proof}

  Meanwhile, for a semi-discrete function $w$ (continuous in space and discrete in time), the following norms are defined: 
\begin{equation} 
   \nrm{ w  }_{\ell^\infty (0, T^*; H^k)} = \max_{0 \leq k \leq K} \nrm{ w^k }_{H^k}, \quad 
   \mbox{for any integer $k \ge 0$ } . 
\end{equation}

Finally, we state the  main result of this paper: 

\begin{thm}
	\label{thm-error-estimate}
For any final time $T > 0$, assume the exact solution $u_\exac$ to the 
GB equation \qref{bsq} given by \qref{Bos-regularity}. 
Denote $u_{\dt,h}$ as the continuous (in space) extension of the fully discrete 
numerical solution given by scheme \ref{numerical}. 
As $\dt, h \to 0$, the following convergence result is valid:
\begin{eqnarray} 
  \left\| u_{\dt,h} - u_\exac \right\|_{\ell^\infty (0, T^*; H^2)} 
  + \nrm{ \psi_{\dt,h} - \psi_\exac }_{\ell^\infty (0, T^*; L^2)} 
  \le C \left( \dt^2 + h^m \right)  ,  
  \label{convergence-u-psi}  
\end{eqnarray}
provided that the time step $\dt$ and the space grid size $h$ are bounded by given constants which are only dependent on the exact solution. 
%\begin{eqnarray} 
%  \dt \le L_1 (T^*, \nu ) ,  \quad h \le L_2 (T^*, \nu) .  
%  \nonumber  %  \label{convergence-1st-dt-1}
%\end{eqnarray} 
Note that the convergence constant in (\ref{convergence-u-psi}) 
also depend on the exact solution as well as $T$. 
%Suppose $U$, $U_N$, $\Psi$ and $\Psi_N$ are as in the  theorem \ref{thm1}.  Define $\tilde{u}_{i}^n := U_N^n\left(h\cdot i\right)-u_{i}^n$ and $\tilde{\psi}_{i}^n := \Psi_N^n\left(h\cdot i\right)-\psi_{i}^n$, where $u^n_{i},\, \psi^n_{i} \in {\mathcal C}_{\overline{n}}$ are the $n^{\rm th}$ periodic solutions of (\ref{numerical}), with $u^0_{i} := U^0_{i}$,  $u^{-1}_{i}=u^0_{i}$ and $\psi^0_{i}=0$.  Then
%	\begin{equation}\label{(6.39)}
%     \| \tilde{\psi}^n\|_{L^{2}}+\parallel\tilde{u}^n\parallel_{H^2} \leq C(\Delta t^2+h^m).
%	\end{equation}
%provided $\Delta t$ is sufficiently small, for some $C>0$ that is independent of $h$ and $\Delta t$.
	\end{thm}

\begin{proof}
Subtracting (\ref{numerical}) from (\ref{numerical2}) yields 
\begin{eqnarray} 
 \frac{\tilde{\psi}^{n+1} - \tilde{\psi}^n }{\dt} 
 &=& -\frac12 D_N^4 ( \tilde{u}^{n+1} + \tilde{u}^n ) +\frac12 D_N^2 ( \tilde{u}^{n+1} + \tilde{u}^n )  + \tau_1^n \nonumber \\
&& + D_N^2 \left( \frac32 \tilde{u}^n 
   \sum_{k=0}^{p-1} ( U^n )^k (u^n )^{p-1-k} 
  - \frac12 \tilde{u}^{n-1}  
  \sum_{k=0}^{p-1}  ( U^{n-1} )^k (u^{n-1} )^{p-1-k}  \right) ,  \label{error1}
\\
\frac{\tilde{u}^{n+1} - \tilde{u}^n}{\Delta t}
&=& \frac{\tilde{\psi}^{n+1} + \tilde{\psi}^n }{2} + \Delta t \tau_2^n . 
\label{error2}
\end{eqnarray} 

Also note a $W^{2,\infty}$ bound for the constructed approximate solution
\begin{equation} 
  \nrm{ U_N }_{L^\infty (0,T^*; W^{2,\infty}) } \le C^* ,  \quad \mbox{i.e.}  \, \, \, 
  \left\| U_N^n \right\|_{L^\infty} \le C^* ,    \left\| (U_N)_x^n \right\|_{L^\infty} \le C^* , 
  \left\| (U_N)_{xx}^n \right\|_{L^\infty} \le C^* , 
%  \, \, \forall n \ge 0 ,  
   \label{convergence-2}
\end{equation}
for any $n \ge 0$, which comes from the regularity of the constructed solution.

\noindent
{\bf An a-priori $H^2$ assumption up to time step $t^n$.}  \, \, 
We assume a-priori that the numerical error function (for $u$) has an $H^2$ bound at time steps $t^n$, $t^{n-1}$,  
\begin{equation} 
  \left\| \tilde{u}_N^k \right\|_{H^2} \le 1 ,  \quad 
  \mbox{with $\tilde{u}_N^k = {\cal I}_N \tilde{u}^k$} ,  \quad 
  \mbox{for $k = n, n-1$} , 
  \label{convergence-a priori-1}
\end{equation} 
so that the $H^2$ and $W^{1,\infty}$ bound for the numerical solution (up to $t^n$) is available 
\begin{eqnarray} 
  &&
  \left\|  u_N^k \right\|_{H^2} 
  = \left\| U_N^k - \tilde{u}_N^k \right\|_{H^2} 
  \le \left\| U_N^k \right\|_{H^2} + \left\| \tilde{u}_N^k \right\|_{H^2}  
  \le C^* + 1 := \tilde{C}_0 , \nonumber 
\\
  && 
    \left\|  u_N^k \right\|_{W^{1,\infty} }  
    \le C  \left\|  u_N^k \right\|_{H^2}  
    \le C \tilde{C}_0 := \tilde{C}_1 , 
    \label{convergence-a priori-2}
\end{eqnarray}  
for $k=n, n-1$, with an 1-D Sobolev embedding applied at the final step.

%Taking a discrete inner product with (\ref{error2}) by the error difference function $(\tilde{\psi}^{n+1}-\tilde{\psi}^{n})$  gives

Taking a discrete inner product with (\ref{error1}) by the error difference function $(\tilde{u}^{n+1}-\tilde{u}^{n})$  gives
\begin{equation}\label{inner2}
\begin{split}
&\left\langle \frac{\tilde{\psi}^{n+1} - \tilde{\psi}^{n}}{\Delta t},\tilde{u}^{n+1}-\tilde{u}^{n}\right\rangle
=\left\langle-\frac{1}{2}D_N^4 \left( \tilde{u}^{n+1}+\tilde{u}^n \right),\tilde{u}^{n+1} - \tilde{u}^{n} \right\rangle \\
&+\left\langle D_N^2\left( \frac32 \tilde{u}^n 
   \sum_{k=0}^{p-1} ( U^n )^k (u^n )^{p-1-k} 
  - \frac12 \tilde{u}^{n-1}  
  \sum_{k=0}^{p-1}  ( U^{n-1} )^k (u^{n-1} )^{p-1-k}  \right), 
  \tilde{u}^{n+1} - \tilde{u}^{n} \right\rangle \\
&+\left\langle \frac{1}{2} D_N^2 \left( \tilde{u}^{n+1}+\tilde{u}^n \right),\tilde{u}^{n+1} - \tilde{u}^{n} \right\rangle + \left\langle \tau_1^n, \tilde{u}^{n+1} - \tilde{u}^{n} \right\rangle .\\
\end{split}
\end{equation}
The leading term of (\ref{inner2}) can be analyzed with the help of \qref{error2}: 
\begin{eqnarray}\label{inner1}
&&\left\langle \frac{\tilde{u}^{n+1} - \tilde{u}^n}{\Delta t},\tilde{\psi}^{n+1} - \tilde{\psi}^{n} \right\rangle
= \left\langle \frac{\tilde{\psi}^{n+1}+\tilde{\psi}^{n}}{2}+\Delta t \tau_2^n,\tilde{\psi}^{n+1} - \tilde{\psi}^{n} \right\rangle \nonumber\\
&=&\frac{1}{2}\left( \nrm{\tilde{\psi}^{n+1}}_2^2 - \nrm{\tilde{\psi}^{n}}_2^2 \right) + \Delta t \left\langle \tau_2^n, \tilde{\psi}^{n+1}-\tilde{\psi}^{n} \right\rangle \nonumber\\
&\geq& \frac{1}{2}\left( \nrm{\tilde{\psi}^{n+1}}_2^2 -  \nrm{\tilde{\psi}^{n}}_2^2 \right) - \frac{1}{2} \Delta t\nrm{\tau_2^n}_2^2 - \Delta t\left( \nrm{\tilde{\psi}^{n+1}}_2^2 + \nrm{\tilde{\psi}^{n}}_2^2 \right).
\end{eqnarray}
The first term on the right hand side of (\ref{inner2}) can be estimated as follows.
\begin{equation}\label{first}
\begin{split}
    \left\langle-\frac{1}{2}D_N^4( \tilde{u}^{n+1}+\tilde{u}^n ),\tilde{u}^{n+1} - \tilde{u}^{n} \right\rangle
    &=-\frac{1}{2} \left\langle D_N^2( \tilde{u}^{n+1} + \tilde{u}^n ), 
D_N^2 ( \tilde{u}^{n+1} - \tilde{u}^{n} ) \right\rangle\\
    &=-\frac{1}{2} \left( \nrm{D_N^2 \tilde{u}^{n+1}}_2^2 
- \nrm{D_N^2 \tilde{u}^{n}}_2^2 \right).
\end{split}
\end{equation}
A similar analysis can be applied to the third term on the right hand side of (\ref{inner2})
\begin{equation}\label{third}
\begin{split}
    \left\langle\frac{1}{2} D_N^2 ( \tilde{u}^{n+1} + \tilde{u}^n ),
\tilde{u}^{n+1} - \tilde{u}^{n} \right\rangle
    &=-\frac{1}{2} \left\langle D_N ( \tilde{u}^{n+1} + \tilde{u}^n ), 
D_N (\tilde{u}^{n+1} - \tilde{u}^{n} ) \right\rangle\\    
&=-\frac{1}{2} \left( \nrm{D_N \tilde{u}^{n+1} }_2^2 
- \nrm{D_N \tilde{u}^{n} }_2^2 \right).
\end{split}
\end{equation}
The inner product associated with the truncation error can be handled 
in a straightforward way: 
\begin{eqnarray}  \label{convergence truncation}  
   \left\langle \tau_1^n, \tilde{u}^{n+1} - \tilde{u}^{n} \right\rangle 
  &=& \frac12 \dt \left\langle \tau_1^n, 
  \tilde{\psi}^{n+1} + \tilde{\psi}^{n} \right\rangle 
  +  \dt^2  \left\langle \tau_1^n, \tau_2^n \right\rangle    \nonumber 
\\
   &\le& \frac12 \left( \nrm{ \tilde{\psi}^{n+1} }_2^2 
  + \nrm{ \tilde{\psi}^{n} }_2^2 \right)  
  +  \frac12 \dt \nrm{ \tau_1^n }_2^2  
  + \frac12 \dt^2  \nrm{ \tau_2^n }_2^2 , 
\end{eqnarray} 
with the error equation (\ref{error2}) applied in the first step.  

For nonlinear inner product, we start from the following decomposition of 
the nonlinear term: 
\begin{eqnarray} 
  {\cal NLT} = {\cal NLT}^1 + {\cal NLT}^2 ,  \quad \mbox{with} \, \, \, 
  &&  
    {\cal NLT}^1 = \frac32 \tilde{u}^n 
   \sum_{k=0}^{p-1} ( U^n )^k (u^n )^{p-1-k}  ,  \nonumber 
\\
  && 
  {\cal NLT}^2 = - \frac12 \tilde{u}^{n-1}  
  \sum_{k=0}^{p-1}  ( U^{n-1} )^k (u^{n-1} )^{p-1-k}  . 
  \label{nonlinear decomposition} 
\end{eqnarray} 
For ${\cal NLT}^1$, we observe that each term appearing in its expansion 
can be written as a discrete interpolation form: 
\begin{eqnarray} 
    \tilde{u}^n ( U^n )^k (u^n )^{p-1-k} 
  = {\cal I}  \left( \tilde{u}_N^n ( U_N^n )^k (u_N^n )^{p-1-k}  \right) ,  
  \, \, \,  0 \le k \le p-1 , 
  \label{convergence-nonlinear-1-1} 
\end{eqnarray} 
so that the following equality is valid: 
\begin{eqnarray} 
    \nrm{ D_N^2 \left( \tilde{u}^n ( U^n )^k (u^n )^{p-1-k}  \right) }_2 
  = \nrm{ \partial_x^2 \left( {\cal I}_N  
  \left( \tilde{u}_N^n ( U_N^n )^k (u_N^n )^{p-1-k}  \right)  \right)  }_{L^2}  . 
  \label{convergence-nonlinear-1-2} 
\end{eqnarray} 
On ther other hand, we see that $\tilde{u}_N^n ( U_N^n )^k (u_N^n )^{p-1-k}  
\in {\cal B}^{p N}$ (for each $0 \le k \le p-1$), so that an application of 
Lemma \ref{aliasing error} gives 
\begin{eqnarray} 
   \nrm{ \partial_x^2 \left(  {\cal I}_N  
  \left( \tilde{u}_N^n ( U_N^n )^k (u_N^n )^{p-1-k}  \right)  \right)  }_{L^2}  
  \le \sqrt{p} \nrm{ \tilde{u}_N^n ( U_N^n )^k (u_N^n )^{p-1-k}  }_{H^2}  . 
  \label{convergence-nonlinear-2} 
\end{eqnarray} 
Meanwhile, a detailed expansion for $\partial_x^j 
\left( \tilde{u}_N^n ( U_N^n )^k (u_N^n )^{p-1-k}  \right)$ (for $0 \le j \le 2$) 
implies that 
\begin{eqnarray} 
    \nrm{ \partial_x^j 
  \left( \tilde{u}_N^n ( U_N^n )^k (u_N^n )^{p-1-k}  \right)  }_{L^2} 
  \le  C  \left( \nrm{U_N^n }_{H^2}^{p-1}   
  + \nrm{u_N^n }_{H^2}^{p-1} + 1 \right) \nrm{ \tilde{u}_N^n }_{H^2} ,  
  \quad 0 \le j \le 2 , 
  \label{convergence-nonlinear-3} 
\end{eqnarray} 
with repeated applications of 1-D Sobolev embedding, H\"older inequality 
and Young inequality. Furthermore, a substitution of 
the bound \qref{convergence-2} for the constructed solution $U_N$ 
and the a-priori assumption \qref{convergence-a priori-1} into 
\qref{convergence-nonlinear-2} leads to 
\begin{eqnarray} 
    \nrm{ \tilde{u}_N^n ( U_N^n )^k (u_N^n )^{p-1-k}  }_{H^2} 
  \le  C  \left(  (C^*)^{p-1}   + ( \tilde{C}_1 )^{p-1} + 1 \right) 
  \nrm{ \tilde{u}_N^n }_{H^2} . 
  \label{convergence-nonlinear-4} 
\end{eqnarray}  
In turn, a combination of \qref{convergence-nonlinear-1-2}, 
\qref{convergence-nonlinear-2} and \qref{convergence-nonlinear-4} 
implies that 
\begin{eqnarray} 
    \nrm{ D_N^2 \left( \tilde{u}^n ( U^n )^k (u^n )^{p-1-k}  \right) }_2 
  \le  C  \left(  (C^*)^{p-1}   + ( \tilde{C}_1 )^{p-1} + 1 \right) 
  \nrm{ \tilde{u}_N^n }_{H^2} . 
  \label{convergence-nonlinear-5} 
\end{eqnarray} 
This bound is valid for any $0 \le k \le p-1$. 
As a result, going back to \qref{nonlinear decomposition}, we get 
\begin{eqnarray} 
    \nrm{ D_N^2 \left( {\cal NLT}^1 \right) }_2 
  \le  \tilde{C}_2 \nrm{ \tilde{u}_N^n }_{H^2} ,  \quad \mbox{with}  \, \, \, 
  \tilde{C}_2 =  C  \left(  (C^*)^{p-1}   + ( \tilde{C}_0 )^{p-1} + 1 \right)  . 
  \label{convergence-nonlinear-6-1} 
\end{eqnarray} 
A similar analysis can be performed to ${\cal NLT}^2$ so that we have 
\begin{eqnarray} 
    \nrm{ D_N^2 \left( {\cal NLT}^2 \right) }_2 
  \le  \tilde{C}_2 \nrm{ \tilde{u}_N^{n-1} }_{H^2} . 
  \label{convergence-nonlinear-6-2} 
\end{eqnarray} 
These two estimates in turn lead to 
\begin{eqnarray} 
     \nrm{ D_N^2 \left( {\cal NLT} \right) }_2 
    = \nrm{ D_N^2 \left( {\cal NLT}^1 \right) }_2  
   + \nrm{ D_N^2 \left( {\cal NLT}^2 \right) }_2 
  \le  \tilde{C}_2 \left( \nrm{ \tilde{u}_N^n }_{H^2} 
  +  \nrm{ \tilde{u}_N^{n-1} }_{H^2} \right) . 
  \label{convergence-nonlinear-6-3} 
\end{eqnarray} 
Consequently, the nonlinear inner product can be analyzed as 
\begin{eqnarray} 
  \hspace{-0.35in} && 
    \left\langle D_N^2 \left( {\cal NLT} \right) , 
  \tilde{u}^{n+1} - \tilde{u}^{n} \right\rangle  
  \le   \dt \nrm{ D_N^2 \left( {\cal NLT} \right) }_2  
  \cdot  \nrm{ \frac{\tilde{u}^{n+1} - \tilde{u}^{n} }{\dt} }_2  \nonumber
\\
  \hspace{-0.35in} &\le& 
     \tilde{C}_2 \dt \left( \nrm{ \tilde{u}_N^n }_{H^2} 
  +  \nrm{ \tilde{u}_N^{n-1} }_{H^2} \right)  
  \cdot  \left( \frac12 \left( \nrm{ \tilde{\psi}^{n+1} }_2 
  + \nrm{ \tilde{\psi}^{n} }_2 + \dt \nrm{ \tau_2^n }_2  
  \right)  \right)    \nonumber 
\\
  \hspace{-0.35in} &\le& 
     C \tilde{C}_2 \dt \left( \nrm{ \tilde{u}_N^n }_{H^2}^2  
  +  \nrm{ \tilde{u}_N^{n-1} }_{H^2}^2  
  + \nrm{ \tilde{\psi}^{n+1} }_2^2  
  + \nrm{ \tilde{\psi}^{n} }_2^2  \right) 
  + C \dt^3 \nrm{ \tau_2^n }_2^2  \nonumber 
\\
   \hspace{-0.35in} &\le& 
     C \tilde{C}_2 \dt \left( \nrm{ D_N^2 \tilde{u}^n }_2^2  
  +  \nrm{ D_N^2 \tilde{u}^{n-1} }_2^2  
  + \nrm{ \tilde{\psi}^{n+1} }_2^2  
  + \nrm{ \tilde{\psi}^{n} }_2^2  \right) 
  + C \dt^3 \nrm{ \tau_2^n }_2^2  
  + C \dt h^{2m} , 
  \label{convergence-nonlinear-7} 
\end{eqnarray} 
in which the preliminary estimate \qref{H2 est-1}, given by 
Lemma \ref{lemma:prelim est}, was applied in the last step. 

Therefore, a substitution of (\ref{first}), (\ref{third}), 
\qref{convergence truncation} and (\ref{convergence-nonlinear-7}) 
into (\ref{inner2}) results in 
 \begin{eqnarray}
\tilde{E}^{n+1} - \tilde{E}^n 
  &\leq& \tilde{C}_3 \dt \left( \nrm{ D_N^2 \tilde{u}^{n} }_2^2  
  +  \nrm{ D_N^2 \tilde{u}^{n-1} }_2^2 
  + \nrm{ \tilde{\psi}^{n+1} }_2^2 
  + \nrm{ \tilde{\psi}^n }_2^2 \right)  \nonumber 
\\
  && 
  + C \dt \left( \nrm{\tau_1^n}_2^2 + \nrm{\tau_2^n}_2^2 \right)  \nonumber 
\\
  &\leq& 
   C \dt \left( \tilde{E}^{n} + \tilde{E}^{n+1} \right) 
  + C M^2 ( \dt^2+ h^m )^2 , 
\end{eqnarray}
with $\tilde{C}_3 = C \tilde{C}_2$, with an introduction of 
a modified energy for the error function
\begin{eqnarray*}
 \tilde{E}^n =\frac12 \left( \nrm{ \tilde{\psi}^n }_2^2 
  + \nrm{D_N^2 \tilde{u}^n }_2^2 + \nrm{D_N \tilde{u}^n}_2^2 \right) .
\end{eqnarray*}
As a result, an application of discrete Grownwall inequality gives
\begin{equation}
 \tilde{E}^l \leq \tilde{C}_4  (\dt^2+ h^m)^2 ,  \quad \forall 0 \le l \le K , 
\end{equation}
which is equivalent to the following convergence result: 
\begin{equation}
   \nrm{ \tilde{\psi}^l }_2 + \nrm{ \tilde{u}_N^l }_{H^2}  
  \leq \tilde{C}_4  (\dt^2+ h^m) ,  \quad \forall 0 \le l \le K . 
  \label{convergence-4}  
\end{equation}

\noindent
{\bf Recovery of the $H^2$ a-priori bound 
\qref{convergence-a priori-1}.}   \, \, 
With the help of the $\ell^\infty (0,T; H^2)$ error estimate 
(\ref{convergence-4}) for the variable $u$, we see that the a-priori $H^2$ bound 
(\ref{convergence-a priori-1}) is also valid for the numerical error function 
$\tilde{u}_N$ at time step $t^{n+1}$, provided that 
\begin{eqnarray} 
  \dt \le \left( \tilde{C}_4 \right)^{-\frac12} ,  \quad  
  h \le \left( \tilde{C}_4 \right)^{-\frac{1}{m}} , \quad 
\mbox{with $\tilde{C}_6$ dependent on $T$} .    
  \nonumber  %  \label{convergence-1st-H2-21} 
\end{eqnarray}
This completes the convergence analysis, 
$\ell^\infty (0,T^*; H^2)$ for $u$, and 
$\ell^\infty (0,T^*; \ell^2)$ for $\psi$. 

  Moreover, a combination of \qref{convergence-4} and the classical 
projection \qref{consistency-Phi-3-1} leads to \qref{convergence-u-psi}. 
The proof of Theorem \ref{thm-error-estimate} is finished. 
\end{proof}

%\section{Conclusions}\label{sec-conclusions}
%In this article, we prove the nonlinear stability and convergence of a fully discrete Fourier pseudospectral method for the GB equation.
%We do this in three steps. First, we derive a local truncation error for a finite Fourier projection of the exact solution to the GB equation
%(\ref{bsq}). Then, we derive an estimate of the difference between our numerical solution and the scheme (\ref{numerical}). Finally, wee use the
%triangle inequality to derive our global estimate. Our results improve the known results obtained by  Frutos  {\it{et al.}}.
%
%\section*{Acknowledgements}
%The authors would like to thank the reviewers for their valuable comments and suggestions.
%%%%%%%%%%%%%%%%%%%%%%%%%%%%%%%%%%%%%%%%%%%%%%%%%%%%%%%%%%%%%%%%%%%%%%%%%%%%%%%%%%%%%%%%%%%%%%%%%%%%%%%%%%%%%%%%%%%%%%%%%%%%%%%%%%%%%%

\begin{rem} 
  One well-known challenge in the nonlinear analysis of pseudospectral 
schemes comes from the aliasing errors. 
For the nonlinear error terms appearing in 
\qref{nonlinear decomposition}, it is clear that any classical approach 
would not be able to give a bound for its second order order derivative 
in a pseudospectral set-up. However, with the help of the aliasing error 
control estimate given by Lem.~\ref{aliasing error}, we could obtain 
an estimate for its discrete $H^2$ norm; see the detailed derivations 
in \qref{convergence-nonlinear-1-1}-\qref{convergence-nonlinear-7}. 

  This technique is the key point in the establishment of a high order 
convergence analysis, $\ell^\infty (0,T^*; H^2)$ for $u$, and 
$\ell^\infty (0,T^*; \ell^2)$ for $\psi$.  Without such an aliasing error 
control estimate, only an $\ell^\infty (0,T^*; \ell^2)$ convergence for $u$, 
and $\ell^\infty (0,T^*; \ell^2)$ convergence can be obtained for $\psi$, 
at the theoretical level; see the detailed discussions in an earlier work 
\cite{Fru}. 
In addition, a severe time step constraint, $\dt \le C h^2$, has to be 
imposed to ensure a convergence in that approach, 
compared to the unconditional convergence established in this article. 
\end{rem}

\section{Numerical Results}\label{sec-numerical}
\setcounter{equation}{0}

In this section we perform a numerical accuracy check for the fully discrete pseudospectral scheme \qref{numerical}. Similar to \cite{Fru}, the exact solitary wave solution of the GB equation (with $p=2$) is given by 
\begin{equation} 
  u_\exac (x,t) = - A \mbox{sech}^2 \left( \frac{P}{2} ( x - c_0 t) \right) , 
  \label{AC-1} 
\end{equation} 
in which $0 < P \le 1$. In more detail, the amplitude $A$, the wave speed $c_0$ and the real parameter $P$ satisfy 
\begin{equation} 
  A = \frac{3 P^2}{2} ,  \quad  c = \left( 1 - P^2 \right)^{1/2} . 
  \label{AC-2} 
\end{equation} 

Since the exact profile \qref{AC-1} decays exponentially as $| x | \to \infty$, it is natural to apply Fourier pseudospectral approximation on an interval $(-L, L)$, with $L$ large enough. In this numerical experiment, we set the computational domain as $\Omega = (-40, 40)$.  A moderate amplitude $A=0.5$ is chosen in the test. 
%and the exact profile for the phase variable is set to be 

\subsection{Spectral convergence in space}

To investigate the accuracy in space, we fix $\dt =10^{-4}$ so that the temporal numerical error is negligible. We compute solutions with grid sizes $N=32$ to $N=128$ in increments of 8, and we solve up to time $T=4$.  The following numerical errors at this final time
\begin{equation} 
  \nrm{ \psi - \psi_\exac}_2 ,  \quad \mbox{and} \quad  
  \nrm{ D_N^2 ( u - u_\exac ) }_2 , 
  \label{AC-3} 
\end{equation} 
are presented in Fig.~\ref{fig1}. The spatial spectral accuracy is apparently observed for both $u$ and $\psi=u_t$. Due to the fixed time step $\dt = 10^{-4}$, a saturation of spectral accuracy appears with an increasing $N$. 

	\begin{figure}
	\begin{center}
\includegraphics[width=4.0in]{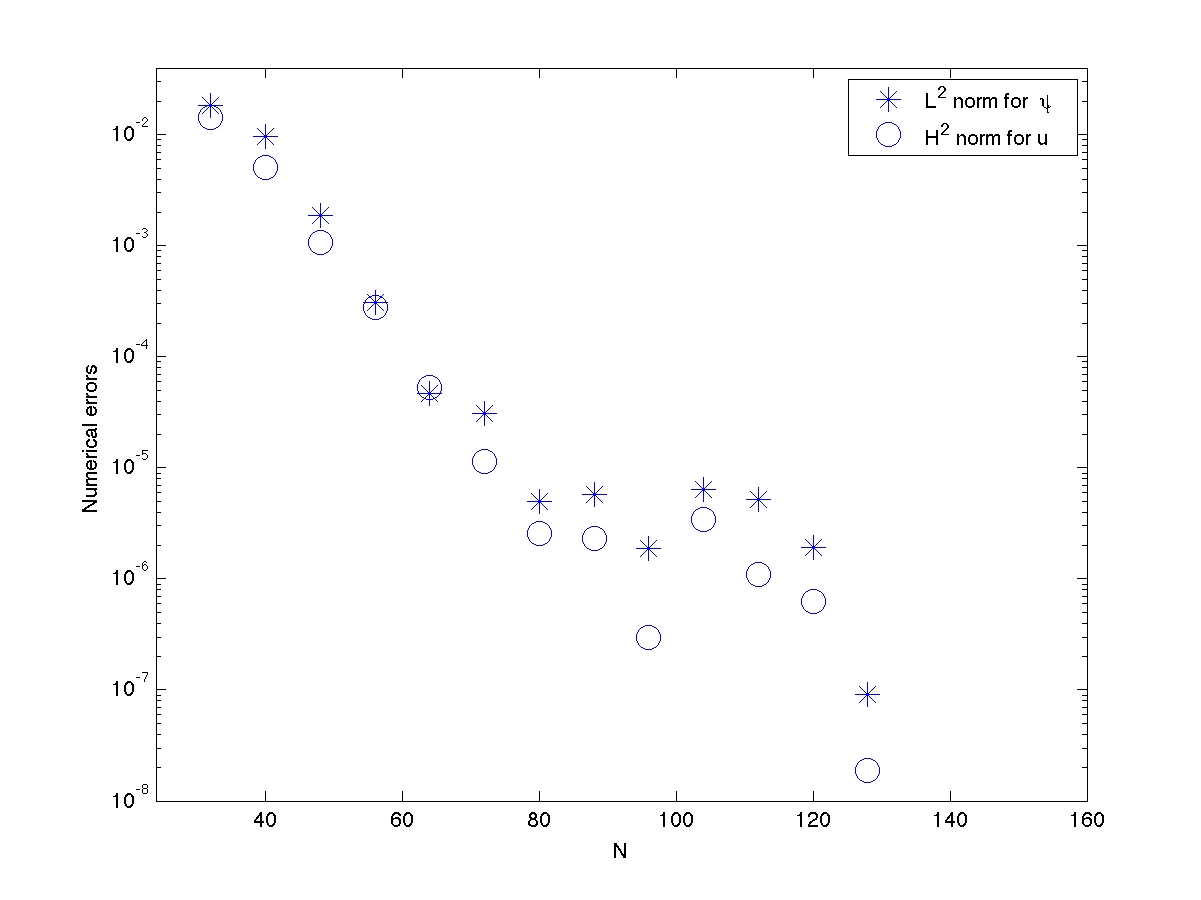}
	\end{center}
\caption{Discrete $L^2$ numerical errors for $\psi=u_t$ and $H^2$ numerical errors for $u$ at $T=4.0$, plotted versus $N$, the number of spatial grid point, for the fully discrete pseudospectral scheme (\ref{numerical}).  The time step size is fixed as $\dt = 10^{-4}$. An apparent spatial spectral accuracy is observed for both variables.}     
	\label{fig1}
	\end{figure}

\subsection{Second order convergence in time}

To explore the temporal accuracy, we fix the spatial resolution as $N=512$ so that the numerical error is dominated by the temporal ones. We compute solutions with a sequence of time step sizes, $\dt = \frac{T}{N_K}$, with $N_K=100$ to $N_K=1000$ in increments of 100, and $T=4$. Fig.~\ref{fig2} shows the discrete $L^2$ and $H^2$ norms of the errors between the numerical and exact solutions, for $\psi = u_t$ and $u$, respectively. A clear second order accuracy is observed for both variables.

	\begin{figure}
	\begin{center}
\includegraphics[width=4.0in]{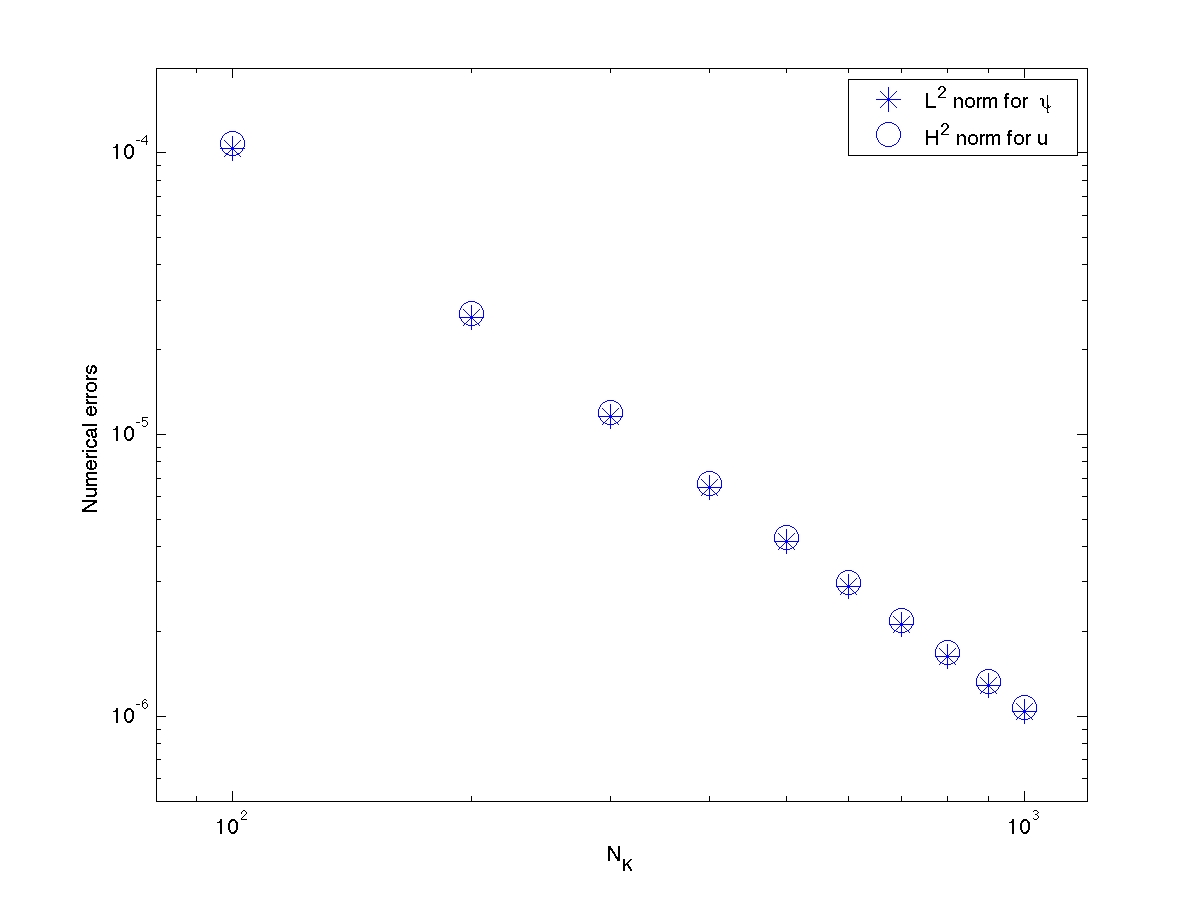}
	\end{center}
\caption{Discrete $L^2$ numerical errors for $\psi=u_t$ and $H^2$ numerical errors for $u$ at $T=4.0$, plotted versus $N_K$, the number of time steps, for the fully discrete pseudospectral scheme (\ref{numerical}).  The spatial resolution is fixed as $N=512$. The data lie roughly on curves $CN_K^{-2}$, for appropriate choices of $C$, confirming the full second-order temporal accuracy of the proposed scheme.}     
	\label{fig2}
	\end{figure}

\section{Conclusion Remarks}\label{sec-conclusions}
In this article, we propose a fully discrete Fourier pseudospectral scheme for the GB equation \qref{bsq} with second order temporal accuracy. The nonlinear stability and convergence analysis are provided in detail. In particular, with the help of an aliasing error control estimate (given by Lem.~\ref{aliasing error}), an $\ell^\infty (0,T^*; H^2)$ error estimate for $u$ and $\ell^\infty (0,T^*; \ell^2)$ error estimate for $\psi=u_t$ are derived. Moreover, an introduction of an intermediate variable $\psi$ greatly improves the numerical stability condition; an unconditional convergence (for the time step $\dt$ in terms of the spatial grid size $h$) is established in this article, compared with a severe time step constraint $\dt \le C h^2$, reported in an earlier literature \cite{Fru}. A simple numerical experiment also verifies this unconditional convergence, second order accurate in time and spectrally accurate in space.

\section*{Acknowledgements}
The authors greatly appreciate many helpful discussions with Panayotis Kevrekidis, in particular for his insightful suggestion and comments.
This work is supported in part by the the Air Force Office of Scientific Research FA-9550-12-1-0224 (S.~Gottlieb), NSF DMS-1115420 (C.~Wang), NSFC 11271281 (C.~Wang).

\end{document}